\documentclass[11pt,reqno,a4paper]{amsart}
\usepackage{a4wide,fullpage}
\usepackage{graphicx,amsmath,amssymb,amsfonts,amsthm,enumitem,bm,xcolor}
\usepackage{cleveref}

\usepackage[normalem]{ulem}

\usepackage{url}

\usepackage{enumitem}

\usepackage{color}
\usepackage{comment}

\usepackage{tabularx}
\usepackage{calc}

\usepackage{longtable,booktabs,array,caption,rotating}
\newcommand{\N}{\mathbb{N}}
\newcommand{\R}{\mathbb{R}}

\newcommand{\s}{{\sigma}}

\renewcommand{\a}{\alpha}
\renewcommand{\b}{\beta}

\renewcommand{\d}{{\delta}}
\newcommand{\g}{\gamma}

\renewcommand{\l}{\lambda}

\renewcommand{\k}{\kappa}

\newcommand{\z}{\zeta}

\renewcommand{\(}{\left\(}
\renewcommand{\)}{\right\)}

\newcommand{\pa}[2]{\left(\frac{#1}{#2}\right)}

\numberwithin{equation}{section}
\theoremstyle{plain}
\newtheorem{theorem}{Theorem}[section]
\newtheorem{lemma}[theorem]{Lemma}

\newtheorem*{remark*}{Remark}
\newtheorem*{remarks}{Remarks}

\newtheorem*{example*}{Example}

\newtheorem{corollary}[theorem]{Corollary}
\newtheorem{proposition}[theorem]{Proposition}

\theoremstyle{definition} 
\newcounter{specialstatement}

\newtheorem{conjecture}[specialstatement]{Conjecture}

\numberwithin{equation}{section}

\renewcommand{\binom}[2]{\left(\begin{smallmatrix}#1\\\\#2\end{smallmatrix}\right)}
\newcommand{\smallbinom}[2]{(\begin{smallmatrix}#1\\#2\end{smallmatrix})}

\setlist[enumerate]{leftmargin=*,label=\rm{(\arabic*)}}

\makeatletter
\@namedef{subjclassname@2020}{%
	\textup{2020} Mathematics Subject Classification}
\makeatother

\allowdisplaybreaks

\newif\ifdefs

\defstrue

\setlist[itemize]{noitemsep, topsep=0pt}
\ifdefs
\else
\excludecomment{discussion}
\excludecomment{extradetails}
\fi

\title{The Jensen--P\'olya program and inequalities for arithmetic sequences}

\author{Koustav Banerjee}
\author{Kathrin Bringmann}
\address{University of Cologne, Department of Mathematics and Computer Science, Weyertal 86-90, 50931 Cologne, Germany}
\email{kbanerj1@uni-koeln.de}
\email{kbringma@math.uni-koeln.de}
\author{Larry Rolen}
\address{Department of Mathematics, 1420 Stevenson Center, Vanderbilt University, Nashville, TN 37240}
\email{larry.rolen@vanderbilt.edu}

\allowdisplaybreaks

\begin{document}
	\maketitle

\begin{abstract}
We develop an analytic framework to study hyperbolicity of Jensen polynomials and related inequalities for arithmetic sequences with general asymptotic growth. Our results provide partial converses to the Hermite–Jensen phenomenon of Griffin, Ono, Zagier, and the third author. As applications we prove several conjectures concerning log-concavity, higher-order Turán inequalities, Laguerre inequalities, infinite log-concavity, and Toeplitz determinants for partition-type functions.
\end{abstract}

\section{Introduction and statement of results}\hspace{0 cm}

\subsection{Introduction}  In this paper, we develop a general analytic framework for studying inequalities of arithmetic sequences with asymptotic conditions. This framework yields partial converses to the Hermite–Jensen phenomenon and leads to several applications to partition-type functions. As an application we resolve several conjectures in the literature.  

The {\it partition function} $p(n)$ counts the number of weakly decreasing sequences of positive integers summing to $n$. Hardy and Ramanujan \cite{HR} famously proved that
\begin{equation}\label{pnasymp} 
 p(n)\sim \frac{1}{4\sqrt{3}}\frac{e^{\sqrt{\frac{2}3 n}}}{n}\quad\quad \text{ as }n\rightarrow\infty.
 \end{equation}
 Here, we study asymptotics of polynomials built out of a large class of sequences. We explore applications to partition-type functions, as well as many functions without analytic structure like modularity. 
We are especially interested in inequalities. This topic has seen a large outpouring of work in recent years. A sequence $\{a(n)\}_{n\ge 0}$ of real numbers is {\it log-concave} for $n$ if
\begin{equation}\label{log-concave}
a^2(n)-a(n-1)a(n+1)\geq0.
\end{equation}
  Nicolas \cite{Nicolas}, and later DeSalvo and Pak \cite{DP15}, proved that $\{p(n)\}_{n\ge 26}$ is log-concave. Chen \cite{Chen1} reinterpreted this in terms of real-rootedness of certain degree two polynomials. 
  A sequence $\{a(n)\}_{n\ge 0}$ of real numbers satisfies the {\it higher order Tur\'{a}n inequality}\footnote{In the literature this is also called {\it third-order Tur\'{a}n inequality.}} for $n\in\N$ if
  \begin{equation}\label{HigherOrderTuranDef}
  \begin{aligned}
  4\left(a^2(n)-a(n-1)a(n+1)\right)&\left(a^2(n+1)-a(n)a(n+2)\right)\\
  &-\left(a(n)a(n+1)-a(n-1)a(n+2)\right)^2\geq 0.
  \end{aligned}
  \end{equation} 
Chen, Jia, and Wang \cite[Theorem 1.4]{Chen2} showed that $\{p(n)\}_{n\ge 95}$ satisfies \eqref{HigherOrderTuranDef}.  
 
The {\it Jensen polynomials} of degree $m\ge 2$ associated to $\{a(n)\}_{n\ge 0}$ are defined by
\begin{equation}\label{Jensendef}
J^{m,n}_a(x):=\sum_{j=0}^{m}\binom{m}{j}a(n+j)x^j.
\end{equation}
 The roots of $J^{2,n}_p(x)$ are real iff $p(n+1)$ is log-concave. Moreover, the roots of $J^{3,n}_p(x)$ are real iff $p(n+1)$ satisfies the higher order Tur\'an inequalities (see \cite{Chen2}). We say that a polynomial $f(x)\in\R[x]$ is {\it hyperbolic} if it has only real roots. Nicolas' and DeSalvo--Pak's result can be reformulated as stating that $J^{2,n}_p$ is hyperbolic for  $n\geq25$. Chen, Jia, and Wang built on the log-concavity and higher-order Tur\'an inequalities for $p$ in \cite[Conjecture 1.5]{Chen2} to conjecture that for each $m$ there exists a constant $N_m$ such that $J^{m,n}_p(x)$ is hyperbolic if $n\geq N_m$. Griffin, Ono, Zagier, and the third author \cite[Theorem 5]{GORZ} proved the conjecture of Chen, Jia, and Wang by defining a renormalized version of the polynomials and showing that these tend to Hermite polynomials.  Before recalling the methods of \cite{GORZ} and our extensions of them, we require some notation.

\subsection{Assumptions} 
 Throughout, $\{a(n)\}_{n\ge 0}$ denotes a sequence of positive numbers. Our  main results  refer to the following assumptions, which are used in different combinations. 
 \begin{enumerate}[label=(C\arabic*)]
  \item \label{GenGORZThm1assump1} For $\a_1,\a_2\in \mathbb{R}$, $\a_1\le j\le a_2$, there exist $\{A(n)\}_{n\ge 0}$, $\{\d(n)\}_{n\ge 0}$ such that
 \begin{equation*}
 \hspace{2 cm}\log\left(\frac{a(n+j)}{a(n)}\right)=A(n)j-\d^2(n)j^2+o\left(1\right)\ \ \ \ (\text{as}\ n\to \infty).
 \end{equation*}
  \item \label{GenGORZThm1assump2} We have $\lim_{n\to \infty}\d(n)=0$ and $\d(n)>0$ for $n\!\gg\!1$.
 \item \label{genthmassump2} We have $\lim_{n\to \infty}A(n)=0$.
	\item \label{InflogconcaveThmassump1} For $\a_1,\a_2\in \mathbb{R}$ and  $\a_1\le j\le a_2$, there exist $\{B(n)\}_{n\ge 0}$ and $\{\b(n)\}_{n\ge 0}$ such that  
	\begin{equation*}
	\log\left(\frac{\d(n+j)}{\d(n)}\right)=B(n)j+\b^2(n)j^2+o\left(1\right)\ \ \ \ (\text{as}\ n\to \infty).
	\end{equation*}
	\item \label{InflogconcaveThmassump2} We have $\lim_{n\to \infty}B(n)= 0$.
	\item \label{InflogconcaveThmassump3} We have $\b(n)=o(\d(n))$ as $n\to \infty$.
\end{enumerate}
 We consider sequences $\{a_d(n)\}_{n\ge 0}$ ($d\in \mathbb{N}$), whose asymptotic growth is of the shape 
\begin{equation}\label{genexample}
a_d(n)\sim c_1(d)\frac{\exp\left(\sum_{\l\in \mathcal{S}}A_{\l}(d)n^{\l}\right)}{n^{c_2(d) }}\ \ \ \ (\text{as}\ n\to \infty)
\end{equation}
 which includes many partition-type functions. Here we assume the following:\footnote{Throughout the paper, for $\{a_d(n)\}_{n\ge 0}$ (as in \eqref{genexample}), we always assume that it satisfies \ref{genexampleassump1}\textendash\ref{genexampleassump4}.}
\begin{enumerate}[label=(C\arabic*), start=7,  labelwidth=\widthof{(C99)}, 
	leftmargin=\labelwidth+\labelsep, 
	align=left]
		\item \label{genexampleassump1} For $\l\in \mathcal{S}$, we have $A_{\l}(d)\neq 0$.
	\item \label{genexampleassump2} The set $\mathcal{S}\subset\mathbb{Q}^{+}\cap (0,1)$ is finite.
	\item \label{genexampleassump3} We have $A_{\l^{*}}(d)>0$ with $\l^{*}:=\underset{\l \in \mathcal{S}}{\max}\{\l \}$.
	\item \label{genexampleassump4} We have $c_1(d)>0$ and $c_2(d)\in \mathbb{R}$.
\end{enumerate}

\subsection{Statement of results} Our first result is a key modification of \cite[Theorem 3]{GORZ} and provides a general criterion guaranteeing hyperbolicity of Jensen polynomials for sequences satisfying the asymptotic conditions \ref{GenGORZThm1assump1} and \ref{GenGORZThm1assump2}.

\begin{theorem}\label{GenGORZcor1}
 If $\{a(n)\}_{n\ge 0}$ satisfies {\rm{\ref{GenGORZThm1assump1}}} with $0\le j\le m$ and {\rm{\ref{GenGORZThm1assump2}}}, then $J^{m,n}_{a}$ is hyperbolic for $n\gg 1$.
\end{theorem}

 \Cref{GenGORZcor1} gives a general structural result for Jensen polynomials. We next investigate partial converses of this phenomenon. Our main goal is to develop a general analytic framework that partially reverses the Hermite–Jensen phenomenon of \cite{GORZ}. Roughly speaking, we show that eventual log-concavity together with mild asymptotic conditions forces Jensen polynomials to become hyperbolic. The starting point of this paper is to establish partial converses of Theorem~\ref{GenGORZcor1}. Specifically, we ask to what extent does eventual log-concavity imply asymptotic shapes of the type considered\footnote{ For the correct shape of the asymptotics stated in \cite[equation (4)]{GORZ}, we refer the reader to \cite{GORZ}.} in \cite[(15)]{GORZ}, and hence the Hermite--Jensen phenomenon underlying Theorem~\ref{GenGORZcor1}. In many natural examples, this indeed holds. We show such partial converses by extending the methods of \cite{GORZ}.  Our new methods give a framework to prove new inequalities in a number of crucial cases. In Lemma~\ref{sec:Jensenlem1}, we show that the corresponding Jensen polynomials are hyperbolic for fixed degree and $n\!\gg\!1$. This allows us to resolve several conjectures. 
 
 Before describing the general framework, we give two applications. For $k\ge 2$, let the {\it $k$-regular overpartitions} $\overline{p}_k(n)$ denote the number of overpartitions of $n$ such that no part is divisible by $k$.  Here {\it overpartitions} of $n$, denoted by $\overline{p}(n)$, count partitions of $n$ where the first occurrence of a part may be overlined. The generating function of $\overline{p}_k(n)$ is, using the {\it $q$-Pochhammer symbol} $(a;q)_n\!:=\!\prod_{j=0}^{n-1}\left(1\!-\!a q^j\right)$ (for $n \in \mathbb{N}_0\cup \{\infty\}$), given by 
\begin{equation}\label{genfunc1}
P_k(q):=\sum_{n\ge 0}\overline{p}_k(n)q^n:=\frac{(-q)_{\infty}\left(q^k;q^k\right)_{\infty}}{(q)_{\infty}\left(-q^k;q^k\right)_{\infty}}.
\end{equation}

 Peng, Zhang, and Zhong \cite[Conjecture 5.1]{PZZ} made the following conjecture. 
\begin{conjecture}\label{conj0}
 {\it For $k\ge 2$, $\overline{p}_k(n)$ satisfies \eqref{log-concave} and \eqref{HigherOrderTuranDef} for $n\gg 1$.} 
\end{conjecture}
For $k\in \mathbb{N}$, Andrews and Paule \cite{AP} introduced broken $k$-diamond partitions (see \cite[Definition 4]{AP}). The counting function is denoted by $b_k(n)$ with generating function 
\begin{equation}\label{genfunc2}
B_k(q):=\sum_{n\ge 0}b_k(n)q^n=\frac{\left(-q\right)_{\infty}}{\left(q\right)^2_{\infty}\left(-q^{2k+1};q^{2k+1}\right)_{\infty}}.
\end{equation}
Dong, Ji, and Jia \cite[p. 614]{DJJ} made the following conjecture.
\begin{conjecture}\label{conj1}
{\it 	For $k\in \mathbb{N}_{\ge 3}$ and $m\in \mathbb{N}$, the Jensen polynomials $J^{m,n}_{b_k}$ are hyperbolic for $n\!\gg\!1$.} 
\end{conjecture}
\begin{theorem}\label{ConjThm1} 
Conjectures \ref{conj0} and \ref{conj1} hold. 
\end{theorem}

 We now describe our general framework (see \Cref{secgen} for more details and additional examples) for studying inequalities among terms of arithmetic sequences with asymptotic conditions \ref{GenGORZThm1assump1}--\ref{genthmassump2}. The following theorem is the main technical tool underlying many of the applications in this paper. For $T\in \mathbb{N}, M_{\ell},\rho_{\ell}(m)\in \mathbb{N}$ and $s_{\ell}(k,m)\in \mathbb{Z}$ with $\ell\in\{1,2\}$, we define
\begin{equation}\label{gendefn}
P^{[M_{\ell}]}_{\rho_{\ell},s_{\ell}}(a)(n):=\sum_{m=1}^{M_{\ell}}\rho_{\ell}(m)\prod_{k=1}^{T}a(n+s_{\ell}(k,m)),\ I^{[\bm{M}]}_{\bm{\rho},\bm{s}}(a)(n):=\frac{P^{[M_1]}_{\rho_1,s_1}(a)(n)-P^{[M_2]}_{\rho_2,s_2}(a)(n)}{a^{T}(n)}.
\end{equation}
Setting $u_1:=1$ and $u_T:=2$ if $T\ge 2$, we express the asymptotics of $I^{[\bm{M}]}_{\bm{\rho},\bm{s}}(a)(n)$ in terms of
\begin{equation}\label{newdef3}
D_t:=\sum_{\ell=1}^{2}(-1)^{\ell+1}\sum_{m=1}^{M_{\ell}}\rho_{\ell}(m) \left(\hspace{0.05 cm} \sum_{k=1}^{T}s^{u_T}_{\ell}(k,m)\right)^t,  
\end{equation}
\begin{equation}\label{newdef2}
\hspace{-3.5 cm}\mathcal{M}(n):=\begin{cases}
-\d^2(n)&\ \ \text{if}\ T\ge 2,\\
A(n)&\ \ \text{if}\ \ T=1,
\end{cases}
\end{equation} 
 where $t\in \mathbb{N}_0$ and $\{A(n)\}_{n\ge 0}$ and $\{\d(n)\}_{n\ge 0}$ are as in \ref{GenGORZThm1assump1}.  We assume the following:
 \begin{enumerate}[label=(C\arabic*), start=11,  labelwidth=\widthof{(C99)}, 
 	leftmargin=\labelwidth+\labelsep, 
 	align=left]
 \item  \label{assump0}	There exists $N\in \mathbb{N}$ such that for all $0\le t\le N-1$, $D_t=0$.
 \item \label{assump1} We have $D_N\neq 0$.
 \end{enumerate}
Define 
\begin{equation*}
\mathcal{C}:=\max\hspace{0.05 cm}\{|s_{\ell}(k,m)|:1\le \ell\le 2, 1\le k\le T, 1\le m\le M_{\ell}\}.
\end{equation*}
 The following theorem is one of the technical backbones of the paper.
\begin{theorem}\label{genthm}
	Let $\{a(n)\}_{n\ge 0}$ satisfy {\rm{\ref{GenGORZThm1assump1}}} with $|j|\le \mathcal{C}$, {\rm{\ref{GenGORZThm1assump2}}}, and {\rm{\ref{genthmassump2}}} and assume {\rm{\ref{assump0}}} and {\rm{\ref{assump1}}}. Then we have 
	\begin{align*}
	&I^{[\bm{M}]}_{\bm{\rho},\bm{s}}(a)(n)\sim \frac{D_N}{N!}\mathcal{M}^N(n)\ \ \left(\text{as}\ n\to\infty\right).
	\end{align*}
\end{theorem}

As an application of \Cref{genthm}, we determine the asymptotics of $\mathcal{D}^{[\bm{M}]}_{\bm{\rho},\bm{s}}(a_d)(n)$, as $n\to \infty$.

\begin{corollary}\label{Corgenthm}
	Let $\{a_d(n)\}_{n\ge 0}$ be as in \eqref{genexample} and assume {\rm{\ref{assump0}}} and {\rm{\ref{assump1}}}. 
	\begin{enumerate}
		\item \label{Corgenthmeqn1}	For $T\ge 2$, we have, as $n\to \infty$,
		\[
		\mathcal{D}^{[\bm{M}]}_{\bm{\rho},\bm{s}}(a_d)(n)\sim \frac{\left(-\!\l^*(1-\l^*)\hspace{0.05 cm}A_{\l^*}(d)\right)^N\! D_N\hspace{0.05 cm}c^T_1(d)}{2^N N!}\frac{\exp\left(T\sum_{\l\in \mathcal{S}}A_{\l}(d)n^{\l}\right)}{n^{Tc_2(d)+\left(2-\l^*\right)N}}.
		\]
		\item \label{Corgenthmeqn2} For $T=1$, we have, as $n\to \infty$,
		\[
		\mathcal{D}^{[\bm{M}]}_{\bm{\rho},\bm{s}}(a_d)(n)\sim \frac{\left(\l^*A_{\l^*}(d)\right)^N\! D_N\hspace{0.05 cm}c_1(d)}{N! }\frac{\exp\left(\sum_{\l\in \mathcal{S}}A_{\l}(d)n^{\l}\right)}{n^{c_2(d)+\left(1-\l^*\right)N}}.
		\]
	\end{enumerate}
\end{corollary}
The results above provide a general framework that allows us to treat a variety of inequalities for partition-type functions in a unified way. In particular, we resolve several conjectures in the literature, which we now describe.

A sequence $\{a(n)\}_{n\ge 0}$ satisfies the {\it Briggs inequality} if
\begin{equation}\label{Briggsineq}
a^2(n)\left(a^2(n)-a(n-1)a(n+1)\right)> a^2(n-1)\left(a^2(n+1)-a(n)a(n+2)\right).
\end{equation}

Denote by $\{p_k(n)\}_{n\ge 0}$ the number of {\it $k$-regular partitions} of $n$, i.e., no part is divisible by $k$. 

\begin{conjecture}{\rm{(\cite[Conjecture 3.11]{LZ})}}\label{LZconj}
{\it  For $k\in \mathbb{N}_{\ge 2}$, $p_k(n)$ and $\overline{p}_k(n)$ satisfy \eqref{Briggsineq} for $n\!\gg\!1$.}
\end{conjecture} 

In this paper, we settle this conjecture.

\begin{theorem}\label{Briggsptnoverptnk}
Conjecture \ref{LZconj} is true. 
\end{theorem}

A sequence satisfies the {\it Laguerre inequality of order $r$} for $n$ if $L_r(a)(n)\geq0$, where for $r\in\mathbb{N}$,
\begin{align}\label{Lagineq}
\hspace{-0.2 cm}L_r(a)(n)\!&:=\!\frac 12\!\sum_{m=0}^{2r}\!(-1)^{m+r}\!\binom{2r}{m}\!a(n+m)a(n+2r-m).
\end{align}
For $n\!\in\!\mathbb{N}$, the {\it double factorial} is $n!!:=\prod_{k=0}^{\lceil\frac n2\rceil-1}(n-2k)$. We prove the following general result. 
\begin{theorem}\label{newlem1}
 Let $\{a(n)\}_{n\ge 0}$ satisfy {\rm{\ref{GenGORZThm1assump1}}} with $|j|\le 2r$, {\rm{\ref{GenGORZThm1assump2}}}, and {\rm{\ref{genthmassump2}}}. Then we have 
	\[
	L_r(a)(n)\sim (2r-1)!!2^{2r-1} \d^{2r}(n)a^2(n)
	\ \ (\text{as}\ \ n\rightarrow\infty).
	\]
\end{theorem}
 The first author \cite[p. 28]{Ban} proposed the following conjecture. 
\begin{conjecture}\label{Banconj}
{\it We have},\footnote{The factor $n^{-\frac{3r}{2}}$ was missing in \cite{Ban}.} {\it as} $n\to \infty$,
\[
L_r(p)(n)\sim \frac{(2r-1)!!\pi^r }{2^{\frac r2+5} 3^{\frac r2+1}}\frac{e^{2\pi\sqrt{\frac{2}3 n}}}{n^{\frac{3r}{2}+2}}.
\]
\end{conjecture} 
Moreover, for $\a>0$, the {\it fractional partitions of $n$}, denoted by $p_{\a}(n)$, are defined via 
\begin{equation*}
\sum_{n\geq0} p_{\a}(n)q^n:=\prod_{n\geq1}\frac{1}{\left(1-q^n\right)^{\a}}.
\end{equation*}
Mukherjee \cite[Conjecture 1.6]{M} proposed the following. 
\begin{conjecture}\label{Mconj1}
{\it We have, as} $n\to \infty$,
\[
L_r\left(p_{24}\right)(n)\sim (2r-1)!!2^{r-2}\pi^r\frac{e^{8\pi\sqrt n}}{n^{\frac{3r+27}2}}.
\]	
\end{conjecture}

\begin{theorem}\label{ex1}
Conjecture \ref{Banconj} and Conjecture \ref{Mconj1} are true.
\end{theorem}

Next, we state the conjecture of Yang \cite[Conjecture 6.2]{Yang1} concerning $b_k$ defined in \eqref{genfunc2}.

\begin{conjecture}\label{brokenLagconj1}
{\it For $k\in\{1,2\}$ and $1\le r\le 14$, $\{b_k(n)\}_{n\ge N_{b_k}(r)}$ satisfies}\footnote{The values of $N_{b_k}(r)$  for $k\in\{1,2\}$, $1\le r\le 14$ were stated in \cite[Conjecture 6.2]{Yang1}.} \eqref{Lagineq}. 
\end{conjecture}

We make partial progress towards Conjecture \ref{brokenLagconj1} by proving the following theorem.

\begin{theorem}\label{brokenLagthm1}
 For $k,r\in \mathbb{N}$, we have, as $n\to \infty$, 
	\[
	L_r(b_k)(n)\sim  
	\frac{(2r-1)!!\pi^r (5k+2)^{\frac{r+3}{2}} }{2^{\frac{r+13}{2}} 3^{\frac{r+3}{2}} (2k+1)^{\frac{r+3}{2}} }\frac{e^{2\pi\sqrt{\frac{2(5k+2)}{3(2k+1)}n}}}{n^{\frac{3r+5}{2}}}.
	\]	
\end{theorem}

We next define the {\it difference operator} on a sequence $\{a(n)\}_{n \geq 0}$ by $$\Delta^{[1]}(a)(n):=a(n)-a(n-1).$$ The {\it $r$-fold difference operator} is given, for $r\geq2$, by
\begin{equation*}
\Delta^{[r]}(a)(n):=\Delta^{[1]}\left(\Delta^{[r-1]}(a)\right)(n).
\end{equation*} 
Gomez, Males, and the third author \cite{gomez2022second} generalized $\Delta$ by introducing a shifted operator 
\begin{equation}\label{deltadef}
\Delta^{[1]}_{h}(a)(n):=a(n)-a(n-h)\ \ \text{and}\ \ \Delta^{[r]}_{h}(a)(n):=\Delta^{[1]}_{h}\left(\Delta^{[r-1]}_{h}(a)\right)(n)\ \ (h\in \mathbb{N}).
\end{equation}
 We prove asymptotics for such difference operators applied to $\{a(n)\}_{n\ge 0}$.
\begin{theorem}\label{newsec4lem1}	
	Let $\{a(n)\}_{n\ge 0}$ satisfy {\rm{\ref{GenGORZThm1assump1}}} with $|j|\le h r$, {\rm{\ref{GenGORZThm1assump2}}}, and {\rm{\ref{genthmassump2}}}. For $r,h\in \mathbb{N}$, 
	\[
	\Delta^{[r]}_h(a)(n)\sim h^rA^r(n) a(n)\ \ \text{ as}\ \ n\rightarrow\infty.
	\]
\end{theorem}
 A {\it plane
	partition} of size $n$ is a two-dimensional array of $\pi_{j,k}\in \mathbb{N}_0$ satisfying $\sum_{j,k}\pi_{j,k}=n$, $\pi_{j,k}\ge \pi_{j,k+1}$, and $\pi_{j,k}\ge \pi_{j+1,k}$ for $j,k\in \mathbb{N}$. We denote the number of plane partitions of $n$ by $\textup{pp}(n)$. Mukherjee \cite[Problem 4.2]{Mshift} asked for asymptotics of $\Delta^{[r]}_{\ell}\left(\textup{pp}\right)(n)$. Here we find these. 

\begin{theorem}\label{mshiftthm}
As $n\to \infty$, we have 
	\[
	\Delta^{[r]}_{h}\left(\textup{pp}\right)(n)\sim C_1\left(\frac{2h A}{3}\right)^r\frac{ e^{A n^{\frac 23}} }{n^{\frac r3+\frac{25}{36}}},
	\]
with $C_1:=\frac{\z(3)^{\frac{7}{36}}e^{\z'(-1)}}{2^{\frac{11}{36}}\sqrt{3\pi}}$, $A:=\frac{3\z(3)^{\frac 13}}{2^{\frac 23}}$, where $\z$ is the Riemann zeta function.
\end{theorem}
 Yang \cite[Conjecture 5.2 (1)]{Yang} proposed the following:
 \begin{conjecture}\label{Yangconj}
{\it For $1\le r\le 5$, $\{\Delta^{[r]}(p)(n)\}_{n\ge T_p(r)}$ and $\{\Delta^{[r]}(\overline{p})(n)\}_{n\ge T_{\overline{p}}(r)}$ satisfy}\footnote{The values of $T_p(r)$ and  $T_{\overline{p}}(r)$ were given in \cite[Table 5]{Yang}.} \eqref{HigherOrderTuranDef}.
\end{conjecture}

We make partial progress towards  Conjecture \ref{Yangconj}.

\begin{theorem}\label{Jensencor1}
	Let $r\in \mathbb{N}$. Then $\Delta^{[r]}(p)(n)$ and $\Delta^{[r]}(\overline{p})(n)$ satisfy \eqref{HigherOrderTuranDef} for $n\!\gg\!1$.	
\end{theorem}

Yang \cite[Conjecture 5.1 (1)]{Yang} also made the following conjecture.
\begin{conjecture}\label{Yangconj1}
{\it For $1\le m\le 11$ and $1\le r\le 5$, $\{\Delta^{[r]}(p)(n)\}_{n\ge L_p(r,m)}$ and}\ \ $\{\Delta^{[r]}(\overline{p})(n)\}_{n\ge L_{\overline{p}}(r,m)}$ {\it satisfy}\footnote{The values for $(L_{p}(r,m))_{\substack{1\le m\le 11\\1\le r\le 5}}$ and $(L_{\overline{p}}(r,m))_{\substack{1\le m\le 11\\1\le r\le 5}}$ were stated in \cite[Table 1 and 2]{Yang}.} \eqref{Lagineq}.
\end{conjecture}

 We make partial progress towards  Conjecture \ref{Yangconj1}.

\begin{theorem}\label{Lagthm1}
	For $m, r\in \mathbb{N}$, we have, as $n\to \infty$,
	\begin{align*}
	L_m\left(\Delta^{[r]}(p)\right)(n)&\sim 
	\frac{(2m-1)!!\pi^{2r+m}}{2^{r+\frac m2+5}3^{r+\frac m2+1} }\frac{e^{2\pi\sqrt{\frac{2}{3}n}}}{n^{r+\frac{3m}{2}+2}},\\
	L_m\left(\Delta^{[r]}(\overline{p})\right)(n)&\sim 
	\frac{(2m-1)!!\pi^{2r+m}}{2^{2r+m+7} }\frac{e^{2\pi\sqrt{n}}}{n^{r+\frac{3m}{2}+2}}.
	\end{align*}
\end{theorem}

As application, we study iterated log-concavity. Define the operator $\mathcal{L}^{[1]}$ on a sequence $\{a(n)\}_{n \geq 0}$ of positive numbers by $\mathcal{L}^{[1]}(a)(n):=\{a^{[1]}(n)\}_{n\geq 0}$ with
\begin{equation}\label{logdef}
a^{[1]}(0):=a^2(0)\ \ \text{and}\ \ a^{[1]}(n):=a^2(n)-a(n+1)a(n-1),\ \text{for}\ n \geq 1.
\end{equation}
Then $\{a(n)\}_{n\ge 0}$ is $\log$-concave iff $\mathcal{L}^{[1]}(a)(n)\ge 0$ for $n\ge 1$. A sequence $\{a(n)\}_{n\ge 0}$ is {\it $r$-$\log$-concave} if $\mathcal{L}^{[r]}(a)(n)\ge 0$, where $\mathcal{L}^{[r]}:=\mathcal{L}^{[1]}\left(\mathcal{L}^{[r-1]}(a)\right)(n).$
A sequence is {\it infinitely $\log$-concave} if it is $r$-log-concave for all $r \in\N$. We show the following general result.
\begin{theorem}\label{InflogconcaveThm}
 Let $\{a(n)\}_{n\ge 0}$ satisfy {\rm{\ref{GenGORZThm1assump1}}} with $|j|\le 2$, {\rm{\ref{GenGORZThm1assump2}}}, {\rm{\ref{genthmassump2}}}, and $\{\d(n)\}_{n\ge 0}$ satisfy {\rm{\ref{InflogconcaveThmassump1}}} with $j=\pm 1$, {\rm{\ref{InflogconcaveThmassump2}}}, and {\rm{\ref{InflogconcaveThmassump3}}}. We have, as $n\to \infty$, 
	\[
	\mathcal{L}^{[r]}(a)(n) \sim 2^{2^{r+1}-r-2}\d(n)^{2^{r+1}-2} a^{2^r}\!(n).
	\]
\end{theorem}
In particular, we settle the following conjecture\footnote{ We make a minor correction of \cite[Conjecture 7.19]{Ban}.} of the first author \cite[Conjecture 7.19]{Ban}. 
\begin{conjecture}\label{Banconj1}
{\it We have, as} $n\to \infty$, 
\[
\mathcal{L}^{[r]}(p)(n)\sim \frac{ \pi^{2^r-1}}{2^{5\cdot 2^{r-1}+r-\frac 12} 3^{2^r-\frac 12}}\frac{ e^{2^r\pi\sqrt{\frac{2}3 n}}}{n^{2^r+3\cdot 2^{r-1}-\frac 32}}. 
\]	
\end{conjecture}

\begin{theorem}\label{inflogthm1} 
Conjecture \ref{Banconj1} is true.
\end{theorem}

	The Jensen--P\'olya program, which launched the subject of this paper, arose from P\'olya's proof that the Riemann Hypothesis is equivalent to the Riemann $\Xi$-function $\Xi(s)\!:=\xi(\frac 12+is)$ (here $\xi(s):=\frac{s(s-1)}{2}\pi^{-\frac s2}\Gamma(\frac s2)\z(s)$) lying in the {\it Laguerre}--{\it P\'olya class}, the space of entire functions which are  local limits of hyperbolic polynomials.  This connection inspired Chen, Jia, and Wang to reinterpret log-concavity and other inequalities for the partition function from the lens of Jensen polynomials.
	Schoenberg \cite{Schoenberg} related the Laguerre--P\'olya class to the theory of totally-positive functions, and the slightly stronger condition of P\'olya frequency functions. Recently, Gröchenig \cite{Groechenig} synthesized this and other related function theory to show that the Riemann Hypothesis is true if $\int_\R\Xi(u+\frac 12)^{-1}e^{-ixu}du$ is a P\'olya frequency function. Such total positivity is  closely related\footnote{See also \cite{Katkova, Nuttall} for more on the relation between total positivity and the Riemann Hypothesis.} to requiring that certain Toeplitz matrices have minors which are positive (see \cite{Groechenig}). Here a {\it Toeplitz matrix} of order $r$ with shifted entries is defined by $(a(n-k+\ell))_{1\le k,\ell \le r}$. This inspires us to consider asymptotics of Toeplitz matrices of combinatorial seqences.  Specifically, we now turn to the determinant of the Toeplitz matrix of order $r$ with shifted entries $(a(n-k+\ell))_{1\le k,\ell \le r}$. Set $\operatorname{T}^{[r]}_a(n):=\det(a(n-k+\ell)_{1\le k,\ell\le r})$. Our general result is the following. 
\begin{theorem}\label{detthm}
	Let $\{a(n)\}_{n\ge 0}$ satisfy {\rm{\ref{GenGORZThm1assump1}}} with $|j|\le r-1$ $(r\in \mathbb{N}_{\ge 2})$ and {\rm{\ref{GenGORZThm1assump2}}}. 
	Then we have 
	\[
	\operatorname{T}^{[r]}_a(n)\sim \frac{(r-1)!^r2^{\frac{r(r-1)}{2}}}{\prod_{j=2}^{r-1}j^{j}}\d^{r(r-1)}(n)a^{r}(n)\ \ (\text{as}\ n\to \infty). 
	\]
\end{theorem}
Jia and Wang \cite[Conjecture 1.7]{JiaWang} made the following conjecture for $p(n)$.
\begin{conjecture}\label{Detconj1}
{\it 	For $k\in \mathbb{N}_{\ge 2}$, we have $\operatorname{T}^{[r]}_p(n)>0$ for $n\!\gg\!1$}.	
\end{conjecture}
We settle this conjecture by proving the following theorem.
\begin{theorem}\label{detpn}
	As $n\to \infty$, we have 
	\[
	\operatorname{T}^{[r]}_{p}(n)\sim 
	\frac{(r-1)!^r\pi^{\frac{r(r-1)}{2}}}{2^{\frac{3r^2+5r}{4}}3^{\frac{r(r+1)}{4}}\prod_{j=2}^{r-1}j^{j}}\frac{e^{\pi r\sqrt{\frac{2}{3}n}}}{n^{\frac{3r^2+r}{4}}}.
	\]
\end{theorem}

Mukherjee \cite[Problem 4.1]{MDet} asked about positivity of $T^{[r]}_{\overline{p}}(n)$ for $r\in \mathbb{N}_{\ge 4}$ and $n\!\gg\!1$. We give an affirmative answer by proving the following theorem.
\begin{theorem}\label{detopn}
	As $n\to \infty$, we have 
	\[
	\operatorname{T}^{[r]}_{\overline{p}}(n)\sim 
	\frac{(r-1)!^r\pi^{\frac{r(r-1)}{2}}}{2^{r^2+2r}\prod_{j=2}^{r-1}j^{j}}\frac{e^{\pi r\sqrt{n}}}{n^{\frac{3r^2+r}{4}}}.
	\]
\end{theorem}
 We also consider the following conjectures of Yang. 
\begin{conjecture}\cite[Conjecture 6.3]{Yang1}\label{Detconj3}
{\it For $k\!\in\!\{1,2\}$ and $1\!\le\!r\!\le\!14$, $\operatorname{T}^{[r]}_{b_k}(n)>0$ for}\hspace{0.05 cm}\footnote{The values of $(N_{b_k}(r))_{\substack{1\le k\le 2\\1\le r\le 14}}$ were given in \cite[Conjecture 6.3]{Yang1}.} $n\ge N_{b_k}(r)$.	
\end{conjecture}
\begin{conjecture}\cite[Conjecture 5.1 (2)]{Yang}\label{Detconj4}
{\it For $1\le r\le 11$ and $1\le k\le 5$, $\operatorname{T}^{[r]}_{\Delta^{[k]}(p)}(n), \operatorname{T}^{[r]}_{\Delta^{[k]}(\overline{p})}(n)>0$ for}\footnote{The values of $(D_p(k,r))_{\substack{1\le r\le 11\\1\le k\le 5}}$ and $(D_{\overline{p}}(k,r))_{\substack{1\le r\le 11\\1\le k\le 5}}$ were given in \cite[Table 3 and Table 4]{Yang}.} {\it $n\ge D_p(k,r)$ and $n\ge D_{\overline{p}}(k,r)$, respectively, except if both $k\ge 3$ and $r$ are odd}.
\end{conjecture}
We make partial progress on these conjectures by proving the following theorem. 
\begin{theorem}\label{detbpn}
As $n\to \infty$, we have 
\begin{align*}
\operatorname{T}^{[r]}_{b_k}(n)&\sim \frac{(r-1)!^r\pi^{\frac{r(r-1)}{2}} (5k+2)^{\frac{r^2}{4}+\frac r2}}{2^{\frac{3r^2}{4}+2r}3^{\frac{r^2}{4}+\frac r2}(2k+1)^{\frac{r^2}{4}+\frac r2}\prod_{j=2}^{r-1}j^{j}}
\frac{e^{\pi r\sqrt{\frac{2(5k+2)}{3(2k+1)}n}}}{n^{\frac{3r^2}{4}+\frac r2}},\\
\operatorname{T}^{[r]}_{\Delta^{[k]}(p)}(n)&\sim \frac{(r-1)!^r\pi^{\frac{r(r-1)}{2}+kr}}{2^{\frac{3r^2+\left(2k+5\right)r}{4}}3^{\frac{r^2+(2k+1)r}{4}}\prod_{j=2}^{r-1}j^{j}} 
\frac{e^{\pi r\sqrt{\frac{2}{3}n}}}{n^{\frac{3r^2+(2k+1)r}{4}}},\\
\operatorname{T}^{[r]}_{\Delta^{[k]}(\overline{p})}(n)&\sim \frac{(r-1)!^r\pi^{\frac{r(r-1)}{2}+kr}}{2^{r^2+(k+2)r}\prod_{j=2}^{r-1}j^{j}} \frac{e^{\pi r\sqrt{n}}}{n^{\frac{3r^2+(2k+1)r}{4}}}.
\end{align*}
\end{theorem}

 The {\it $k$-fold} ($k\in \mathbb{N}_0$) {\it application} of $L_1$ (see \eqref{Lagineq}) on $\{a(n)\}_{n\ge 0}$ is defined as
\begin{equation}\label{itlagdef}
L^{[0]}_1(a)(n):=a(n)\ \ \text{ and }\ \  L^{[k]}_1:=L^{[1]}_1\left(L^{[k-1]}_1(a)\right)(n)\ \ \text{ for }\ \ k\in \mathbb{N}.
\end{equation} 
\begin{theorem}\label{itlagthm1}
 Let $\{a(n)\}_{n\ge 0}$ satisfy {\rm{\ref{GenGORZThm1assump1}}} with $|j|\le 2r$, {\rm{\ref{GenGORZThm1assump2}}}, {\rm{\ref{genthmassump2}}} and $\{\d(n)\}_{n\ge 0}$ satisfy {\rm{\ref{InflogconcaveThmassump1}}} with $j=\pm 1$, {\rm{\ref{InflogconcaveThmassump2}}}, and {\rm{\ref{InflogconcaveThmassump3}}}. Then, for $k\in \mathbb{N}_0$ and $r\in \mathbb{N}$, we have, as $n\to \infty$, 
	\[
	L^{[k]}_1\left(L_r(a)\right)(n)\sim (2r-1)!!^{2^k}2^{ 2^{k+1}\left(r+1\right)-k-3}\d(n)^{2^{k+1}\left(r+1\right)-2} a^{2^{k+1}}(n). 
	\]
\end{theorem}

 We next consider a conjecture of Dou, Tang, and Wang \cite[Conjecture 6.6]{DTW}. To prove positivity of an infinite class of polynomials depending on a sequence $\{a(n)\}_{n\ge 0}$, at first seems that one needs to write explicit expressions for each of these polynomials. As the authors of \cite{DTW} pointed out below (6.5) loc. cit.: ``it seems to be difficult to represent $L^{[k]}_1\left(L_r(a)\right)(n)$ in a unified form $\ldots$ the number of terms increases rapidly.'' However, we overcome this problem by using induction on  iterated operators to obtain asymptotics and positivity results without closed-form expressions. Our next result concerns $L^{[k]}_1\left(L_r(a)\right)(n)$, however, as we see, the same principle applies to other expressions such as infinite log-concavity and determinants of Toeplitz matrices. 
 
 Returning to the specific conjecture at hand, recall that Andrews \cite{SPT} defined  $\textup{spt}(n)$ as the total number of appearances of the smallest part in partitions of $n$. 

\begin{conjecture}\label{DTWConj}
{\it Let $a(n)\in \{p(n),\overline{p}(n),\textup{spt}(n)\}$. For $n\!\gg\!1$ and $k,r\in\mathbb{N}$, $	L^{[k]}_1\left(L_r(a)\right)(n)>0$.}
\end{conjecture}

We settle this conjecture by proving the following theorem.

\begin{theorem}\label{itlagthm3}
	As $n\to \infty$, we have
	\begin{align*}
	L^{[k]}_1\left(L_r(p)\right)(n)&\sim \frac{(2r-1)!!^{2^k}\pi^{2^k(r+1)-1}}{2^{2^{k-1}\left(r+9\right)+k+\frac 12}3^{2^{k-1}\left(r+3\right)-\frac 12}}\frac{e^{2^{k+1}\pi\sqrt{\frac{2}{3}n}}}{n^{2^{k-1}\left(3r+7\right)-\frac 32}},\\
	L^{[k]}_1\left(L_r(\overline{p})\right)(n)&\sim \frac{(2r-1)!!^{2^k}\pi^{2^k(r+1)-1}}{2^{2^k\left(r+7\right)+k}}\frac{e^{2^{k+1}\pi\sqrt{n}}}{n^{2^{k-1}\left(3r+7\right)-\frac 32}},\\ 
	L^{[k]}_1\left(L_r(\textup{spt})\right)(n)&\sim \frac{(2r-1)!!^{2^k}\pi^{2^k(r-1)-1}}{2^{2^{k-1}\left(r+7\right)+k+\frac 12}3^{2^{k-1}\left(r+1\right)-\frac 12}}\frac{e^{2^{k+1}\pi\sqrt{\frac{2}{3}n}}}{n^{2^{k-1}\left(3r+5\right)-\frac 32}}.
	\end{align*}	
 In particular, Conjecture \ref{DTWConj} is true.
\end{theorem}

 We next consider sequences $\{g_d(n)\}_{n\ge 0}$ ($d\in \mathbb{N}$) satisfying
\begin{equation}\label{loggenexample}
g_d(n)\sim h_1(d) n^{h_2(d)}\ \ \text{ as}\ \ n\to \infty,
\end{equation}
where $h_1(d), h_2(d)>0$. Jia \cite[Conjecture 1.8]{Jia} proposed the following conjecture.
\begin{conjecture}\label{brokenlogdiffconj}
{\it For $k, r\!\in\!\mathbb{N}$, there exists $N_k(r)\in\N$ such that 
	$
	(-1)^{r+1}\Delta^{[r]}(\log(b_k))(n)>0$ for $n\ge N_k(r)$.}
\end{conjecture}
We settle this conjecture by proving the following theorem.
\begin{theorem}\label{logdiffconj1}
	Let $r,k\in \mathbb{N}$. Then we have, as $n\to \infty$, 
	\[
	(-1)^{r+1}\Delta^{[r]}\left(\log(b_k)\right)(n)\sim 
	\frac{\pi(2r-3)!!\sqrt{5k+2}}{2^{r-\frac 12} \sqrt{3\left(2k+1\right)}}n^{\frac 12-r}.
	\]
\end{theorem}

The paper is organized as follows. In \Cref{sec0}, we recall some basic facts. We prove \Cref{GenGORZcor1} and \Cref{ConjThm1} in \Cref{sec:GenJensen}, Theorems \ref{genthm}\textendash\ref{Lagineq} in \Cref{secgen}, and Theorems \ref{InflogconcaveThm}\textendash\ref{logdiffconj1} in \Cref{newsec2}. Finally in \Cref{concluding}, we discuss problems for future research.

\section*{Acknowledgements}
The first author has received funding from the Deutsche Forschungsgemeinschaft (DFG) entitled ``Eine allgemeiner Ansatz für die Hyperbolizität von Jensen-Polynomen und kombinatorische Ungleichungen" (grant No. BA 9185/1-1, 574654157). The second author has received funding from the European Research Council (ERC) under the European Unions Horizon 2020 research and innovation programme (grant agreement No. 101001179). The third author was supported by a grant from the Simons Foundation (853830, LR). The authors thank Don Zagier for his insightful remarks.

\section{Preliminaries}\label{sec0}

\subsection{A Tauberian Theorem} We adopt the setup from \cite[Theorem 1.1]{BJM}. 
For $\Delta\ge 0$, define 
\[
R_{\Delta}:=\{x+iy:x,y\in \mathbb{R}, x>0, |y|\le \Delta x\}.
\]
\begin{proposition}\label{sec0prop1}
	Let $B(q)=\sum_{n\ge 0}b(n)q^n$ have radius of convergence at least one with $\{b(n)\}_{n\ge 0}$ non-negative and weakly increasing. Assume $\l,\b,\g\in \mathbb{R}$ with $\g>0$ exists with
	\begin{equation}\label{sec0prop1eqn1}
	B\left(e^{-t}\right)\sim \l t^{\b}e^{\frac \g t}\ \text{ as}\ t\to 0^{+},\quad B\left(e^{-z}\right)\ll |z|^\b e^{\frac{\g}{|z|}}\ \ \text{ as}\ \ z\to 0,
	\end{equation}
	with the later condition holding in each region $R_{\Delta}$ for $\Delta\ge 0$. Then we have
	\[
	b(n)\sim \frac{\l \g^{\frac \b2+\frac 14}}{2\sqrt{\pi}n^{\frac \b2+\frac 34}}e^{2\sqrt{\g n}}\ \text{ as}\ n\to \infty.
	\]
\end{proposition}
 We need the following well-known lemma that is not hard to prove.
\begin{lemma}\label{elem}
Let $\{a(n)\}_{n\ge 0}$ satisfy $a(0)=1$ and set $A(q):=\sum_{n\ge 0}a(n)q^n$. If $(1-q)A(q)$ has non-negative Fourier coefficients, then $\{a(n)\}_{n\ge 0}$ is weakly increasing.
\end{lemma}

Moreover we require the asymptotics of $(q;q)_{\infty}$. 
\begin{lemma}\label{sec0lem1}
	Let $q=e^{-z}$ and $\Delta\ge 0$. Then, as $z\to 0$ in $R_{\Delta}$, we have
	\[
	(q;q)_{\infty}\sim \sqrt{\frac{2\pi}{z}}e^{-\frac{\pi^2}{6z}}.
	\]
\end{lemma}

\subsection{Hermite and Jensen polynomials} The {\it Hermite polynomial} of degree $m$ is 
\begin{equation}\label{Hermitedef}
H_m(x):=m!\sum_{r=0}^{\left\lfloor\frac m2\right\rfloor}\frac{(-1)^r}{r!\left(m-2r\right)!}(2x)^{m-2r}.
\end{equation}
The main theorem of Griffin, Ono, Zagier, and the third author is the following. 
\begin{theorem}\cite[Theorem 3]{GORZ}\label{GORZThm1}
	Let $\{a(n)\}_{n\ge 0}, \{A(n)\}_{n\ge 0}, \{\d(n)\}_{n\ge 0}$ be sequences of positive numbers with $\lim_{n\to \infty}\d(n)=0$. Assume that for some $m\in \mathbb{N}$ and all $0\le j\le m$, we have
	\begin{equation*}
	\log\left(\frac{a(n+j)}{a(n)}\right)=A(n)j-\d^2(n)j^2+\sum_{k=3}^{m}A_k(n)j^k+o\left(\d^{m}(n)\right)\ \ \text{ as}\ \ n\to \infty,
	\end{equation*}
 and $A_k(n)=o(\d^k(n))$ as $n\to \infty$. Then we have, locally uniformly in $x$,
	\[
	\lim_{n\to \infty}\frac{J^{m,n}_{a}\left(\frac{\d(n)x-1}{e^{A(n)}}\right)}{a(n)\d^{m}(n)}=H_m\left(\frac x2\right).
	\]
\end{theorem}
\Cref{GORZThm1} implies the following.
\begin{corollary}\cite[Corollary 4]{GORZ}\label{GORZcor1}
	The Jensen polynomials $J^{m,n}_{a}$ for a sequence $\{a(n)\}_{n\ge 0}$ satisfying the conditions of \Cref{GORZThm1} are hyperbolic for $n\!\gg\!1$.
\end{corollary}

 For a polynomial $\mathcal{P}(x)\in \mathbb{R}[x]$, define the {\it discriminant} $\operatorname{\mathbb{D}}(\mathcal{P})\!\!:=\!a^{2m-2}_m\prod_{j< k}(\a_j-\a_k)^2$, where $m$ is the degree of $\mathcal{P}$, $a_m$ is the leading coefficient of $\mathcal{P}$, and $\{\a_k\}_{1\le k\le m}$ are the roots of $\mathcal{P}$. By \cite[Section 6]{Chen1}, $\{a(n)\}_{n\ge 0}$ is log-concave iff $\operatorname{\mathbb{D}}(J^{2,n}_a)\ge 0$ and $J^{3,n}_a$ is hyperbolic iff $\operatorname{\mathbb{D}}(J^{3,n}_a)\ge 0$. Moreover, by \cite[equation (1.2)]{Chen2},  $\{a(n)\}_{n\ge 0}$ satisfies \eqref{HigherOrderTuranDef} iff $\operatorname{\mathbb{D}}(J^{3,{n-1}}_a)\ge 0$.

We require the following elementary identity (see \cite[26.8.4 and 26.8.6]{DLMF}).
\begin{lemma}\label{useful}
For $n,t\in\mathbb{N}$, we have 
\begin{equation*}
\sum_{k=0}^{n}(-1)^k\binom{n}{k}k^t=\begin{cases}
0&\ \text{if}\ \ t<n,\\
(-1)^n n!&\ \text{if}\ \ t=n.
\end{cases}
\end{equation*}
\end{lemma}

\section{Proof of Theorems \ref{GenGORZcor1} and \ref{ConjThm1}}\label{sec:GenJensen}

 In the following theorem, we improve \Cref{GORZThm1}. While our proof shares some similarities with \cite{GORZ}, we expand and combine terms differently. 

\begin{theorem}\label{GenGORZThm1}
	Let $\{a(n)\}_{n\ge 0}$ satisfy {\rm{\ref{GenGORZThm1assump1}}} with $0\le j\le m$ and {\rm{\ref{GenGORZThm1assump2}}}. Then, locally uniformly in $x$, we have
\[
\lim_{n\to \infty}\frac{J^{m,n}_{a}\left(\frac{\d(n)x-1}{e^{A(n)}}\right)}{a(n)\d^{m}(n)}=H_m\left(\frac x2\right).
\]%
\end{theorem}

\begin{proof}
	Using \eqref{Jensendef} and \ref{GenGORZThm1assump1} with $0\le j\le m$, we have  
	\begin{align}\label{Jenseneqn1} 
	&\frac{J^{m,n}_{a}\left(\frac{\d(n)x-1}{e^{A(n)}}\right)}{a(n)\d^{m}(n)}
	=\frac{1+o\left(1\right)}{\d^{m}(n)}\left(\sum_{j=0}^{m}\binom{m}{j}\left(\d(n)x-1\right)^j+\sum_{j=0}^{m}\binom{m}{j}\sum_{r=1}^{\left\lfloor\frac m2\right\rfloor}\frac{(-1)^r\d^{2r}(n)j^{2r}}{r!}\left(\d(n)x-1\right)^j\right.\nonumber\\[-1em]
	&\hspace{7.5 cm}\left.+\sum_{j=0}^{m}\binom{m}{j}\!\!\!\!\!\!\sum_{r\ge \left\lfloor\frac m2\right\rfloor+1}\!\!\!\!\!\frac{(-1)^r\d^{2r}(n)j^{2r}}{r!}\left(\d(n)x-1\right)^j\right)\nonumber\\
	&=\left(1+o\left(1\right)\right)\left(x^{m}+\sum_{r=1}^{\left\lfloor\frac m2\right\rfloor}\frac{(-1)^r\d^{2r-m}(n)}{r!}\sum_{j=0}^{m}\binom{m}{j}j^{2r}\left(\d(n)x-1\right)^j\right.\nonumber\\[-1.2em]
	&\hspace{5 cm}\left.+\d^{-m}(n)\sum_{j=0}^{m}\binom{m}{j}\!\!\!\!\!\!\sum_{r\ge \left\lfloor\frac m2\right\rfloor+1}\!\!\!\!\!\frac{(-1)^r\d^{2r}(n)j^{2r}}{r!}\left(\d(n)x-1\right)^j\right).
	\end{align}

	Next, we expand
	\begin{align}\label{Jenseneqn2}
	&x^m+\sum_{r=1}^{\left\lfloor\frac m2\right\rfloor}\frac{(-1)^r\d^{2r-m}(n)}{r!}\sum_{j=0}^{m}\binom{m}{j}j^{2r}\left(\d(n)x-1\right)^j\nonumber\\
	&\hspace{1 cm}=x^m+\sum_{r=1}^{\left\lfloor\frac m2\right\rfloor}\frac{(-1)^r}{r!}\sum_{\ell=0}^{m-2r}(-1)^{\ell}\d(n)^{2r-m+\ell}x^{\ell}\sum_{j=\ell}^{m}(-1)^j\binom{m}{j}\binom{j}{\ell}j^{2r}\nonumber\\
	&\hspace{2 cm}+\sum_{r=1}^{\left\lfloor\frac m2\right\rfloor}\frac{(-1)^r}{r!}\sum_{\ell=m-2r+1}^{m}(-1)^{\ell}\d(n)^{2r-m+\ell}x^{\ell}\sum_{j=\ell}^{m}(-1)^j\binom{m}{j}\binom{j}{\ell}j^{2r}.
	\end{align}
Using \eqref{Hermitedef} and \Cref{useful}, the first two terms combine as 
\begin{align*}
x^m+\sum_{r=1}^{\left\lfloor\frac m2\right\rfloor}\frac{(-1)^r}{r!}\sum_{\ell=0}^{m-2r}\binom{m}{\ell}\d(n)^{2r-m+\ell}x^{\ell}\sum_{k=0}^{2r}\binom{2r}{k}\ell^{2r-k}\sum_{j=0}^{m-\ell}(-1)^{j}\binom{m-\ell}{j}j^k=H_m\left(\frac x2\right).
\end{align*}
Plugging this into \eqref{Jenseneqn2} and then into \eqref{Jenseneqn1}, we have 
	\begin{align}\label{Jenseneqn5}
	&\frac{J^{m,n}_{a}\left(\frac{\d(n)x-1}{e^{A(n)}}\right)}{a(n)\d^{m}(n)}
	=\left(1+o\left(1\right)\right)\left(H_m\left(\frac x2\right)+S^{[1]}_m(n,x)+S^{[2]}_m(n,x)\right),
	\end{align}
	where 
	\begin{align*}
	S^{[1]}_m(n,x)&:=\sum_{r=1}^{\left\lfloor\frac m2\right\rfloor}\frac{(-1)^r}{r!}\sum_{\ell=m-2r+1}^{m}(-1)^{\ell}\d(n)^{2r-m+\ell}x^{\ell}\sum_{j=\ell}^{m}(-1)^j\binom{m}{j}\binom{j}{\ell}j^{2r},\nonumber\\
	S^{[2]}_m(n,x)&:=\d^{-m}(n)\sum_{j=0}^{m}\binom{m}{j}\sum_{r\ge \left\lfloor\frac m2\right\rfloor+1}\frac{\left(-1\right)^r}{r!}\d^{2r}(n)j^{2r}\left(\d(n)x-1\right)^j.
	\end{align*}
Using \ref{GenGORZThm1assump2}, we then bound 	
	\begin{align}\label{Jenseneqn9}
	\left|S^{[1]}_m(n,x)\right|&\le 
	\d(n)\!\sum_{r=1}^{\left\lfloor\frac m2\right\rfloor}\!\frac{m^{2r}}{r!}\!\sum_{\ell=m-2r+1}^{m}\!\binom{m}{\ell}\!|x|^{\ell}\sum_{j=\ell}^{m}\!\binom{m-\ell}{j-\ell}
	=2^m\d(n)\sum_{r=1}^{\left\lfloor\frac m2\right\rfloor}\!\frac{m^{2r}}{r!}\!\sum_{\ell=m-2r+1}^{m}\!\binom{m}{\ell}\pa{|x|}{2}^{\ell}\nonumber\\
	&
	\le 2^m\left(e^{m^2}-1\right)\left(\left(1+\frac{|x|}{2}\right)^m-1\right)\d(n)=O_m\left(\d(n)\right),
	\end{align}
assuming that $x$ lies in a compact set. Next, using \ref{GenGORZThm1assump2}, we similarly obtain 
	\begin{align*}
\left|S^{[2]}_m(n,x)\right|=O_m\left(\d(n)\right). 
\end{align*}
Plugging this and \eqref{Jenseneqn9} into \eqref{Jenseneqn5}, and letting $n\to \infty$, we obtain the theorem.\qedhere 
\end{proof}

\begin{proof}[Proof of \Cref{GenGORZcor1}]
 It is not hard to see that if a sequence of real polynomials $P^{[d]}_n(x)$ of degree $d$ converges locally uniformly to a real polynomial $P_d(x)$ of degree $d$ with only real simple zeros, then $P^{[d]}_{n}$ has only real simple zeros for $n\!\gg\!1$. Applying this to \Cref{GenGORZThm1} and using the fact that $H_m(x)$ has simple real zeros (see \cite[§18.16(v)]{DLMF}), we obtain the claim.\qedhere 	
\end{proof}

In particular for $m=2$, if $\{a(n)\}_{n\ge 0}$ satisfies \ref{GenGORZThm1assump1} with $0\le j\le 2$ and \ref{GenGORZThm1assump2}, then by \Cref{GenGORZcor1}, it is log-concave for $n\gg 1$. In \Cref{HypEquivthm}, we show that under certain conditions, hyperbolicity of $J^{2,n}_a$ holds iff hyperbolicity of $J^{m,n}_a$ holds for $m\in \mathbb{N}_{\ge 2}$. We require: 
\begin{enumerate}[label=(C\arabic*), start=13]
	\item \label{Nlem1assump2} For $n\!\gg\!1$, $a(n)$ satisfies \eqref{log-concave}.
	\item \label{Nlem1assump3} For all $j\in \mathbb{N}_0$, $L_j:=\lim_{n\to \infty}\frac{a(n+j)}{a(n)}$ is independent of $j$.
\end{enumerate}
\begin{lemma}\label{Nlem1}
 Let $\{a(n)\}_{n\ge 0}$ satisfy {\rm{\ref{Nlem1assump2}}}. 
Then, for all $j\in \mathbb{N}_0$, $L_j=1$.
\end{lemma}
\begin{proof}
Using \ref{Nlem1assump2} and $a(n)>0$ for $n\in \mathbb{N}$,
$\{\frac{a(n+j)}{a(n)}\}_{n\ge 0}$ is positive and weakly decreasing. Thus 
$L_j\ge 0$. By \ref{Nlem1assump3}, $L_j$ is independent of $j$. Computing $L_0=1$ gives the claim.
\end{proof}
 Moreover, we assume the following:

\begin{enumerate}[label=(C\arabic*), start=15]
\item \label{Nlem2assump1} For all $j\in \mathbb{N}_0$, we have 
\begin{equation*}
\log\left(\frac{a(n+j)}{a(n)}\right)=\log\left(1+\psi_n(j)\right)+o(1)
\ \text{ as }\ n\to \infty.
\end{equation*}
\end{enumerate}
Here $\psi_x(y)$ has an asymptotic expansion (as $x\to \infty$) of the form $\psi_x(y)=\sum_{\ell\ge 1}\frac{c_{\ell}(x)}{\ell!}y^{\ell}$ with $\lim_{x\to \infty}c_1(x)=0$, $c_1(x)>0$, and $c_2(x)=o(c^2_1(x))$ both for $x\!\gg\!1$.
\begin{lemma}\label{Nlem2}
If $\{a(n)\}_{n\ge 0}$ satisfies {\rm{\ref{Nlem1assump2}}}\textendash{\rm{\ref{Nlem2assump1}}}, then it satisfies {\rm{\ref{GenGORZThm1assump1}}} with $0\le j\le m$ and {\rm{\ref{GenGORZThm1assump2}}}.
\end{lemma}
\begin{proof}
 By \Cref{Nlem1} and \ref{Nlem2assump1}, $\lim_{n\to \infty}\psi_n(j)=0$ for all $j\in \mathbb{N}_0$. By \ref{Nlem2assump1}, $\lim_{n\to \infty}c_1(n)=\lim_{n\to \infty}c_2(n)=0$.  Using this and \ref{Nlem2assump1}, $\{a(n)\}_{n\ge 0}$ satisfies \ref{GenGORZThm1assump1} with $0\le j\le m$, 
$A(n):=c_1(n)$, and $\d(n):=\sqrt{\frac 12(c^2_1(n)-c_2(n))}$ and by \ref{Nlem2assump1}, $\{\d(n)\}_{n\ge 0}$ satisfies \ref{GenGORZThm1assump2}.\qedhere  
\end{proof}

Using \Cref{Nlem2} and \Cref{GenGORZcor1}, we obtain the following theorem. 

\begin{theorem}\label{HypEquivthm}
Let $\{a(n)\}_{n\ge 0}$ satisfy {\rm{\ref{Nlem1assump3}}} and  {\rm{\ref{Nlem2assump1}}} and assume that, 
 for $n\gg1$,
$J^{2,n}_a(x)$ is hyperbolic. Then $J^{m,n}_a(x)$ is hyperbolic for $m\in \mathbb{N}_{\ge 2}$ and $n\!\gg\!1$.
\end{theorem}

 Next, we need the following elementary lemmas.

\begin{lemma}\label{Fact1}
	If $f(n)\sim g(n)$ as $n\to \infty$, then, for $\a_1\le j\le \a_2$ 
	 with $j$ fixed, 
	\[
	\log\left(\frac{f(n+j)}{f(n)}\right)=\log\left(\frac{g(n+j)}{g(n)}\right)+o(1)\ \ \text{as}\ \ n\to\infty.
	\]
\end{lemma}

\begin{lemma}\label{fact2}
Let $\{a_d(n)\}_{n\ge 0}$ be as in \eqref{genexample}. 
\begin{enumerate}
	\item \label{fact2part1} The sequence $\{a_d(n)\}_{n\ge 0}$ satisfies {\rm{\ref{GenGORZThm1assump1}}} for all $\a_1\le j\le \a_2$, {\rm{\ref{GenGORZThm1assump2}}}, and {\rm{\ref{genthmassump2}}}.
	\item \label{fact2part2} The sequence $\{\d(n)\}_{n\ge 0}$ satisfies {\rm{\ref{InflogconcaveThmassump1}}} for all $\a_1\le j\le \a_2$, {\rm{\ref{InflogconcaveThmassump2}}}, and {\rm{\ref{InflogconcaveThmassump3}}}. 
\end{enumerate}
\end{lemma}
\begin{proof}
From \eqref{genexample} and \Cref{Fact1}, we have,
\begin{align}\label{neweqn9}
\log\left(\frac{a_d(n+j)}{a_d(n)}\right)
&=A(n)j-\d^2(n)j^2+\sum_{\ell\ge 3}A_{\ell}(n)j^\ell+o(1)\ \ \left(\text{ as}\ n\to \infty\right),
\end{align}
where 
\begin{align*}
A(n)&:=\sum_{\l\in \mathcal{S}}\l A_{\l}(d)n^{\l-1}-\frac{c_2(d)}{n},\quad \d(n):=\frac1{\sqrt2n}\sqrt{\sum_{\l\in \mathcal{S}}\l(1-\l) A_{\l}(d)n^{\l}-c_2(d)},\\
A_{\ell}(n)&:=\sum_{\l\in \mathcal{S}}\binom{\l}{\ell}A_{\l}(d)n^{\l-\ell}+\frac{(-1)^\ell c_2(d)}{\ell n^\ell}.
\end{align*}

\noindent \ref{fact2part1} Using \ref{genexampleassump2} and \ref{genexampleassump3}, we have, as $n\to \infty$,  
\begin{equation}\label{obs4}
A(n)\sim \frac{\l^*A_{\l^*}(d)}{n^{1-\l^*}}\to 0.
\end{equation}
Thus \ref{genthmassump2} is satisfied. From \ref{genexampleassump3}, we obtain $\d(n)>0$ for $n\!\gg\!1$. Using \ref{genexampleassump3}, we have 
\begin{equation}\label{impfact}
\d(n)\sim \sqrt{\frac{\l^*(1-\l^*)A_{\l^*}(d)}{2}}n^{\frac{\l^*}{2}-1}\to 0\ \ \ \ \ \ (\text{as}\ n\to \infty). 
\end{equation}
 Thus \ref{GenGORZThm1assump2} is satisfied. Next, for  $\a_1\le j\le \a_2$, we show that
\begin{align}
\label{jenseneqn1}
\sum_{\ell\ge 3}A_{\ell}(n)j^\ell&=o\left(1\right).
\end{align}
Using $|\smallbinom{\l}{\ell}|\le \l$ for $\ell\in \mathbb{N}_{\ge 3}$ and \ref{genexampleassump2}\textendash\ref{genexampleassump4}, we estimate, for $\ell\in\mathbb{N}_{\ge 3}$, 
\begin{equation}\label{newlem2eqn3a}
|A_{\ell}(n)| \le \left(\l^*\sum_{\l\in \mathcal{S}} |A_{\l}(d)|+\frac{\left|c_2(d)\right|}{3}\right)n^{\l^*-\ell}=:\k_d n^{\l^*-\ell}.
\end{equation}
 Using \eqref{newlem2eqn3a}, we bound, for $n\!\gg\!1$ (and $n\ge 2K$, where $K:=\max\{|\a_1|, |\a_2|\}$), 
\begin{align*}
\left|\sum_{\ell\ge 3}A_{\ell}(n)j^\ell\right|&\le \sum_{\ell\ge 3}\left|A_{\ell}(n)\right|K^\ell
\le 2\k_{d}K^{3}n^{\l^*-3}.
\end{align*}
 Consequently, by \ref{genexampleassump2}, and \ref{genexampleassump3}, we obtain that \eqref{jenseneqn1} holds. Using this and \eqref{neweqn9}, $\{a_d(n)\}_{n\ge 0}$ satisfies \ref{GenGORZThm1assump1} with $\a_1\le j\le \a_2$. This concludes the proof of \ref{fact2part1}. 

\noindent \ref{fact2part2}  By \eqref{impfact} and \Cref{Fact1}, we have 
\begin{align}\label{neweqn11}
\log\left(\frac{\d(n+j)}{\d(n)}\right)
&=B(n)j+\beta^2(n)j^2+\sum_{m\ge 3}B_{m}(n)j^m+o(1),
\end{align}
where
\[
B(n):=\left(\frac{\l^*}{2}-1\right)\frac 1n,\quad \beta(n):=\sqrt{\frac 12\left(1-\frac{\l^*}{2}\right)}\frac 1n,\quad B_m(n)=\left(\frac{\l^*}{2}-1\right)\frac{(-1)^{m+1}}{m n^m}.
\]
We have $\lim_{n\to \infty}B(n)=0$ and $\b(n)=o(\d(n))$. Thus \ref{InflogconcaveThmassump2} and \ref{InflogconcaveThmassump3} are satisfied. Note that, for $\a_1\le j\le \a_2$ and $n\!\gg\!1$, 
\[
\sum_{m\ge 3}B_{m}(n)j^m=o(1).
\] 
Using \eqref{neweqn11}, we see that $\{\d(n)\}_{n\ge 0}$ satisfies \ref{InflogconcaveThmassump1} with $\a_1\le j\le \a_2$.\qedhere 
\end{proof}

 By \Cref{fact2} \ref{fact2part1}, $\{a_d(n)\}_{n\ge 0}$ satisfies \ref{GenGORZThm1assump1} with $0\le j\le m$ and \ref{GenGORZThm1assump2}. Consequently, by \Cref{GenGORZcor1}, we obtain the following.

\begin{lemma}\label{sec:Jensenlem1}
	Let $\{a_d(n)\}_{n\ge 0}$ be as in \eqref{genexample} and $m\ge 2$. Then $J^{m,n}_{a_d}$ is hyperbolic for $n\gg 1$.
\end{lemma}

 Now, we are ready to determine the asymptotics of  $\overline{p}_k(n)$ as $n\to \infty$.

\begin{lemma}\label{overptnlem1}
	Let $k\ge 2$. We have\footnote{\Cref{overptnlem1} corrects typos from \cite[equation (3.18)]{PZZ}.}
	\[
	\overline{p}_k(n)\sim \frac{\left(k-1\right)^{\frac 14}}{2^{\frac 32}k^{\frac 34} }\frac{e^{\pi\sqrt{\frac{k-1}k n}}}{n^{\frac 34}}\ \ \text{ as}\ \ n\to \infty.
	\]
\end{lemma}
\begin{proof} 
We rewrite \eqref{genfunc1} as 
\begin{equation*} 
P_k(q)=\sum_{n\ge 0}\overline{p}_k(n)q^n=\frac{\left(q^2;q^2\right)_{\infty}\left(q^{k};q^{k}\right)^2_{\infty}}{\left(q\right)^2_{\infty}\left(q^{2k};q^{2k}\right)_{\infty}}.
\end{equation*}
By \Cref{sec0lem1}, as $z\to 0$ in $R_{\Delta}$ ($\Delta\ge 0$), 
	\[
	P_k(e^{-z})\sim 
	\frac{e^{\frac{(k-1)\pi^2}{4kz}}}{\sqrt k}.
	\]
	Thus $P_k(q)$ satisfies \eqref{sec0prop1eqn1} with $\l=\frac{1}{\sqrt{k}}$, $\b=0$, and $\g=\frac{(k-1)\pi^2}{4k}$. Next, note that $\overline{p}_k(0)=1$. It remains to show that $\{\overline{p}_k(n)\}_{n\ge 0}$ is weakly increasing. Using \eqref{genfunc1}, we obtain  
	\begin{align*}
\left(1-q\right)P_k(q)
=\frac{\prod_{1\le j\le k-1}\left(-q^j;q^k\right)_{\infty}}{\left(q^{k+1};q^k\right)_{\infty}\prod_{2\le j\le k-1}\left(q^j;q^k\right)_{\infty}}.
	\end{align*}
Thus $(1-q)P_k(q)$ has non-negative Fourier coefficients and thus, by \Cref{elem}, $\overline{p}_k(n)\ge \overline{p}_k(n-1)$ for $n\in \mathbb{N}$. The claim then follows, using \Cref{sec0prop1}.\qedhere  
\end{proof}

Next, we determine the asymptotic main term of $b_k(n)$ as $n\to \infty$. Similar to the proof of \Cref{overptnlem1}, using \eqref{genfunc2}, \Cref{sec0lem1}, \Cref{elem}, and \Cref{sec0prop1}, we obtain the following.

\begin{lemma}\label{brokenlem1}
	For $k\in \mathbb{N}$, we have 
	\[
	b_k(n)\sim \frac{(5k+2)^{\frac 34}}{2^{\frac{11}4} 3^{\frac 34} (2k+1)^{\frac 34}}\frac{e^{\pi\sqrt{\frac{2(5k+2)}{3(2k+1)}n}}}{n^{\frac 54}}\ \ \text{ as}\ \ n\to \infty.
	\]
\end{lemma}

\begin{proof}[Proof of \Cref{ConjThm1}]
	By \Cref{overptnlem1} and \Cref{sec:Jensenlem1}, we conclude the proof of Conjecture \ref{conj0}. By \Cref{brokenlem1} and \Cref{sec:Jensenlem1} of the paper, we obtain Conjecture \ref{conj1}.\qedhere   
\end{proof}

\section{Proofs of Theorems \ref{genthm}\textendash\ref{Lagineq}}\label{secgen}
 
Let $P^{[M_{\ell}]}_{\rho_{\ell},s_{\ell}}(a)(n)$ be as in \eqref{gendefn}. Define
\begin{align}\label{gendefn1}
\mathcal{D}^{[\bm{M}]}_{\bm{\rho},\bm{s}}(a)(n)
&:=P^{[M_1]}_{\rho_1,s_1}(a)(n)-P^{[M_2]}_{\rho_2,s_2}(a)(n),
\end{align}
We next explain how known examples fit into this framework. 
\begin{enumerate}[
	label=(E\arabic*),
	wide,
	labelwidth=!,
	labelindent=0pt
	]

\item {\bf Log-concavity}

\noindent \label{LC}  Shifting $n\mapsto n+1$, 
the left-hand side of \eqref{log-concave} fits into the above framework with $T=2, M_1=M_2=1$, $\rho_1(1)=\rho_2(1)=1$, $s_1(1,1)=
s_1(2,1)=1$, $s_2(1,1)=0$, $s_2(2,1)=2$, and $N=1$.  
\vspace{0.2 cm}
 
\item {\bf Higher order Tur\'{a}n inequality}  

\noindent \label{HT} Expanding the left-hand side of \eqref{HigherOrderTuranDef}, we have
\begin{align*}
\hspace{1 cm}3a^2(n)a^2(n+1)+&6a(n-1)a(n)a(n+1)a(n+2)-4a^3(n)a(n+2)-4a(n-1)a^3(n+1)\\
&\hspace{7 cm}-a^2(n-1)a^2(n+2).
\end{align*}
We are in the above framework with $T=4$, $M_1=2$, $M_2=3$, $\rho_1(1)=3$, $\rho_1(2)=6$, $\rho_2(1)=\rho_2(2)=4$, $\rho_2(3)=1$, $s_1(1,1)=s_1(2,1)=0$, $s_1(3,1)=s_1(4,1)=1$, $s_1(1,2)=-1$,  $s_1(2,2)=0$, $s_1(3,2)=1$, $s_1(4,2)=2$, $s_2(1,1)=s_2(2,1)=s_2(3,1)=0$, $s_2(4,1)=2$, $s_2(1,2)=-1$, $s_2(2,2)=s_2(3,2)=s_2(4,2)=1$, $s_2(1,3)=s_2(2,3)=-1$, $s_2(3,3)=s_2(4,3)=2$, and $N=3$. 

\vspace{0.2 cm}

\item {\bf Laguerre inequality of order two}	

\noindent \label{L2} By \eqref{Lagineq} with $r=2$, we have 
\begin{align*}
L_2(a)(n)=a(n)a(n+4)\!+\!3a^2(n+2)\!-\!4a(n+1)a(n+3).
\end{align*}
 This fits into the above framework with $T=2$, $M_1=2$, $M_2=1$, $\rho_1(1)=1$, $\rho_1(2)=3$, $\rho_2(1)=4$, $s_1(1,1)=0$, $s_1(2,1)=4$, $s_1(1,2)=s_1(2,2)=2$, $s_2(1,1)=1$, $s_2(2,1)=3$, and $N=2$.

\vspace{0.2 cm} 

\item {\bf Briggs' inequality} 

\noindent \label{Briggsineq3} By \cite[equation (3)]{LZ}, to prove \eqref{Briggsineq}, it suffices to show that \eqref{log-concave} holds and 
\begin{equation}\label{Briggsineq2}
B(a)(n):=a^2(n)a(n+1)+a(n-1)a(n)a(n+2)-2a(n-1)a^2(n+1)>0.
\end{equation}
This fits in the above framework with $T=3$, $M_1=2$, $M_2=1$, $\rho_1(1)=\rho_1(2)=1$, $\rho_2(1)=2$, $s_1(1,1)=s_1(2,1)=0$, $s_1(3,1)=1$, $s_1(1,2)=-1$, $s_1(2,2)=0$, $s_1(3,2)=2$, $s_2(1,1)=-1$, $s_2(2,1)=s_2(3,1)=1$, and $N=2$. 
\end{enumerate}

\begin{proof}[Proof of \Cref{genthm}]
  Note that 
 $\mathcal{C}\in \mathbb{N}$. 
For $\ell\in \{1,2\}$, $1\le m\le M_{\ell}$ and $N\in \mathbb{N}$, set 
	\[w_{s_{\ell},m}(n):=-\d^2(n)\sum_{k=1}^{T}s^2_{\ell}(k,m),\quad \phi_N(x):=\sum_{m=1}^{N}\frac{x^m}{m!}.\]
We first assume that $T\ge 2$. Using \eqref{gendefn} and \ref{GenGORZThm1assump1} with $|j|\le \mathcal{C}$, we have, for $\ell\in \{1,2\}$, 
	\begin{align}\label{genthmeqn1}
	\frac{P^{[M_{\ell}]}_{\rho_{\ell},s_{\ell}}(a)(n)}{a^{T}(n)}\!&=\!
\left(1\!+\!o_T\left(1\right)\right)\!\sum_{m=1}^{M_{\ell}}\!\rho_{\ell}(m)e^{A(n)\!\sum_{k=1}^{T}\!s_{\ell}(k,m)}
\left(\!1\!+\!\phi_{N}\!\left(w_{s_{\ell},m}(n)\right)\!+\!\!\!\!\sum_{t\ge N+1}\!\frac{w^t_{s_{\ell},m}(n)}{t!}\!\right).
	\end{align}
	Bounding $w_{s_{\ell},m}(n)\ll_{C,T} \d^2(n)$, 
	we have, for $n\!\gg\!1$,  
	\begin{align*}
	\left|\sum_{t\ge N+1}\frac{w^t_{s_{\ell},m}(n)}{t!}\right|
	\le \frac{\left|w^{N+1}_{s_{\ell},m}(n)\right|}{\left(N+1\right)!}\sum_{t\ge 0}\frac{1}{t!}\ll_{\mathcal{C},T,N}\d^{2N+2}(n).	\end{align*}
Plugging 
this into \eqref{genthmeqn1}, using \ref{genthmassump2} and \eqref{gendefn}, we have, for $n\!\gg\!1$, 
\begin{align}\label{genthmeqn15}
I^{[\bm{M}]}_{\bm{\rho},\bm{s}}(a)(n)=\left(1\!+\!o_{T}\left(1\right)\right)\Bigl(1+O_{\mathcal{C}, T}\left(A(n)\right)\Bigr)\sum_{\ell=1}^{2}T^{[\ell]}_{\bm{M},N}(n),
\end{align}
where 
	\begin{align*}
	T^{[1]}_{\bm{M},N}(n)&:=\sum_{\ell=1}^2 (-1)^{\ell+1}\sum_{m=1}^{M_{\ell}}\!\rho_{\ell}(m)\phi_N\!\left(-\d^2(n)\!\sum_{k=1}^{T}\!s^2_{\ell}(k,m)\right),\\
	T^{[2]}_{\bm{M},N}(n)&:=\Bigl(1+O_{\mathcal{C},T,N}\left(\d^{2N+2}(n)\right)\Bigr)D_0.
	\end{align*}
Using  \ref{assump0}, we have, for $n\!\gg\!1$, 
	\begin{equation}\label{genthmeqn16}
	T^{[2]}_{\bm{M}, N}(n)=0.
	\end{equation}
 Finally, using \eqref{newdef3}, \ref{assump0}, and \ref{assump1}, we have 
	\begin{align*}
	T^{[1]}_{\bm{M},N}(n)=\frac{(-1)^ND_N}{N!}\d^{2N}(n).
	\end{align*}
	Plugging this and \eqref{genthmeqn16} 
	into \eqref{genthmeqn15} and using \ref{genthmassump2} we have, for $T\ge 2$, as $n\to \infty$,
	\begin{align*}
	&I^{[\bm{M}]}_{\bm{\rho},\bm{s}}(a)(n)\sim \frac{(-1)^{N}D_N}{N!}\d^{2N}(n).
	\end{align*}	
 This, together with \eqref{newdef2}, concludes the proof of \Cref{genthm} for $T\ge 2$.

Next we let $T=1$. For $\ell\in \{1,2\}$, using \eqref{gendefn}, \ref{GenGORZThm1assump1} with $|j|\le \mathcal{C}$, \ref{GenGORZThm1assump2}, and \ref{genthmassump2}, we have,
\[
\frac{P^{[M_{\ell}]}_{\rho_{\ell},s_{\ell}}(a)(n)}{a(n)}=\left(1\!+\!o_{\mathcal{C}}\left(1\right)\right)\! \sum_{m=1}^{M_{\ell}}\!\rho_{\ell}(m)\!\left(1\!+\!\phi_N\left(A(n)s_{\ell}(1,m)\right)\!+\!O_{\mathcal{C}}\left(A^{2}(n)\right)\right)\ \ \left(\text{as}\ n\to \infty\right).
\]
Thus, using \eqref{gendefn}, \ref{assump0}, and \eqref{newdef3}, we have, as $n\to \infty$,
\[
I^{[\bm{M}]}_{\bm{\rho},\bm{s}}(a)(n)\sim \frac{D_N}{N!}A^N(n).
\]
 Combining this and using \eqref{newdef2}, the claim follows.\qedhere 
\end{proof}

\begin{proof}[Proof of \Cref{Corgenthm}]
 By \Cref{fact2} \ref{fact2part1}, $\{a_d(n)\}_{n\ge 0}$ satisfies \ref{GenGORZThm1assump1} with $|j|\le \mathcal{C}$, \ref{GenGORZThm1assump2}, \ref{genthmassump2}. Using \eqref{gendefn1}, \eqref{gendefn}, and \Cref{genthm}, we have
 \begin{equation}\label{nclaim3}
\frac{\mathcal{D}^{[\bm{M}]}_{\bm{\rho},\bm{s}}(a_d)(n)}{a^T_d(n)}\sim \frac{D_N}{N!}\mathcal{M}^N(n). 
 \end{equation}

\noindent \ref{Corgenthmeqn1}  By \eqref{nclaim3}, \eqref{newdef2}, and \eqref{impfact}, we have, as $n\to \infty$, 
\[
\frac{\mathcal{D}^{[\bm{M}]}_{\bm{\rho},\bm{s}}(a_d)(n)}{a^T_d(n)}\sim \frac{\left(-\l^*(1-\l^*)A_{\l^*}(d)\right)^N\!D_N }{2^NN!}n^{\left(\l^*-2\right)N}.
\]
Using this and \eqref{genexample}, we conclude the proof of \ref{Corgenthmeqn1}. 

\noindent \ref{Corgenthmeqn2} \hspace{-0.05 cm}  The proof is analogous to that of \ref{Corgenthmeqn1}, using \eqref{nclaim3}, \eqref{newdef2}, \eqref{obs4}, and \eqref{genexample}.\qedhere 
\end{proof}

As an application of \Cref{genthm}, we have the following result.

\begin{corollary}\label{genthmcor1}
Let $\{a_d(n)\}_{n\ge 0}$ be as in \eqref{genexample}. Then $a_d(n)$ satisfies \eqref{Briggsineq} for $n\!\gg\!1$.
\end{corollary}
\begin{proof}
 To prove \eqref{Briggsineq}, it suffices to show that $\{a_d(n)\}_{n\ge 0}$ satisfies \eqref{log-concave} and \eqref{Briggsineq2} for $n\!\gg\!1$. By \Cref{sec:Jensenlem1}, \eqref{log-concave} holds for $\{a_d(n)\}_{n\ge 0}$ for $n\!\gg\!1$.  From \ref{Briggsineq3}, 
  we have, using \eqref{newdef3}, $D_2=8$. 
Moreover, by \Cref{fact2} \ref{fact2part1}, $\{a_d(n)\}_{n\ge 0}$ satisfies \ref{GenGORZThm1assump1} with $|j|\le 2$, \ref{GenGORZThm1assump2}, and \ref{genthmassump2}. Thus by \Cref{Corgenthm} \ref{Corgenthmeqn1}, we have, as $n\to \infty$, 
	\begin{align*}
	B\left(a_d(n)\right)
	&\sim \left(\l^*(1-\l^*)A_{\l^*}(d)\right)^2c^3_1(d)\frac{\exp\left(3\sum_{\l\in \mathcal{S}}A_{\l}(d)n^{\l}\right)}{n^{3c_2(d)+2\left(2-\l^*\right)}}.
	\end{align*}
	This implies that $\{a_d(n)\}_{n\ge 0}$ satisfies \eqref{Briggsineq2} for $n\!\gg\!1$ and thus concludes the proof.\qedhere  
\end{proof}

\begin{proof}[Proof of \Cref{Briggsptnoverptnk}]
 Using the asymptotics of $p_k(n)$ (see \cite[equation (2.1)]{CP}),  \Cref{overptnlem1}, and \Cref{genthmcor1}, we conclude the proof.
\end{proof}

\begin{proof}[Proof of \Cref{newlem1}]
Using \eqref{Lagineq}, we write 
\begin{align*}
&2(-1)^rL_r(a)(n)
=\!\sum_{m=1}^{r+1}\!\binom{2r}{2m-2}\!a(n+2m-2)a(n+2r-2m+2)\nonumber\\[-10 pt]
&\hspace{7 cm}-\sum_{m=1}^{r}\binom{2r}{2m-1}a(n+2m-1)a(n+2r-2m+1).
\end{align*}
 In the notation of \eqref{gendefn1} and \eqref{gendefn}, $T=2$, $M_1=r+1$, $M_2=r$, $\rho_1(m)=\smallbinom{2r}{2m-2}$, $\rho_2(m)=\smallbinom{2r}{2m-1}$, $s_1(1,m)=2m-2$, $s_1(2,m)=2r-2m+2$, $s_2(1,m)=2m-1$, and $s_2(2,m)=2r-2m+1$. 
Then $\mathcal{C}=2r$. Next, we show that \ref{assump0} and \ref{assump1} are satisfied with $N=r$. We have
\begin{align}\label{obs1}
&
D_t=\sum_{\ell=0}^{t}\binom{t}{\ell}\sum_{k=0}^{2t-2\ell}\binom{2t-2\ell}{k}(-2r)^k\sum_{m=0}^{2r}(-1)^m\binom{2r}{m}m^{2t-k}.
\end{align}
 By \Cref{useful}, this vanishes for $t<r$. Using \Cref{useful}, \eqref{obs1} equals $2^r(2r)!$ for $t=r$. Thus \ref{assump0} and \ref{assump1} hold with $N=r$ and $D_r=2^r(2r)!$. This concludes the proof.\qedhere 
\end{proof}

 Consequently, by \Cref{fact2} \ref{fact2part1}, \Cref{newlem1}, \eqref{genexample}, and \eqref{impfact}, we have the following. 

\begin{corollary}\label{newcor1}
Let $\{a_d(n)\}_{n\ge 0}$ be as in \eqref{genexample}. We have, as $n\to \infty$, 
\[
L_r(a_d)(n)\sim (2r-1)!!\frac{c_1^2(d)}{2}\left(2\l^*(1-\l^*)A_{\l^*}(d)\right)^r\frac{\exp\left(2\sum_{\l\in \mathcal{S}}A_{\l}(d)n^{\l}\right)}{n^{2c_2(d)+(2-\l^*)r}}.
\]
\end{corollary}

Next, we determine asymptotics of $L_r(p_{\a})(n)$, as $n\to \infty$.

\begin{corollary}\label{ex2}
	With $\a>0$ we have, as $n\to \infty$, 
	\[
	L_r(p_{\a})(n)\sim \frac{(2r-1)!!\pi^r\a^{\frac{\a+r+1}{2}} }{ 2^{\frac{3\a+r+7}{2}} 3^{\frac{\a+r+1}{2}} }\frac{e^{2\pi\sqrt{\frac{2\a}3 n}}}{n^{\frac{\a+3+3r}{2}}}.
	\]
\end{corollary}

\begin{proof}
	From \cite[Corollary 1.2]{IJT}, we have, as $n\to \infty$, 
	\[
	p_\a(n)\sim
	\frac{1}{\sqrt{2}}\pa{\a}{24}^{\frac{\a+1}{4}}\frac{e^{\pi\sqrt{\frac{2\a}3 n}}}{n^{\frac{\a+3}{4}}}.
	\]
 Applying \Cref{newcor1}, the claim follows.
\end{proof}

\begin{proof}[Proof of \Cref{ex1}]
 Plugging $\a=1$ resp. $\a=24$ into \Cref{ex2} gives the theorem.\qedhere 	
\end{proof}

\begin{proof}[Proof of \Cref{brokenLagthm1}] 
 Plugging \Cref{brokenlem1} into \Cref{newcor1} gives the theorem. \qedhere 
\end{proof}

\begin{proof}[Proof of \Cref{newsec4lem1}]
First, using \eqref{deltadef}, we have, for $r\in \mathbb{N}$,
\begin{align*}
\Delta^{[r]}_h(a)(n)=\sum_{m=1}^{\left\lfloor\frac r2\right\rfloor+1}\binom{r}{2m-2}a\left(n-h(2m-2)\right)-\sum_{m=1}^{\left\lfloor\frac{r+1}{2}\right\rfloor}\binom{r}{2m-1}a\left(n-h(2m-1)\right).
\end{align*}	
 In the notation of \eqref{gendefn} and \eqref{gendefn1}, $T=1$, $M_1=\lfloor\frac{r}{2}\rfloor+1$, $M_2=\lfloor\frac{r+1}{2}\rfloor$, $\rho_1(m)=\smallbinom{r}{2m-2}$, $\rho_2(m)=\smallbinom{r}{2m-1}$, $s_1(1,m)=-h(2m-2)$, and $s_2(1,m)=-h(2m-1)$. Then, $\mathcal{C}=h r$. Using \Cref{useful} and \eqref{newdef3}, \ref{assump0} and \ref{assump1} hold with $N=r$. Moreover, $D_r=h^rr!$. 
  Using this, \Cref{genthm} with $T=1$ and $N=r$, and \eqref{newdef2}, we conclude the proof.\qedhere  
\end{proof}

By \Cref{fact2} \ref{fact2part1}, \Cref{newsec4lem1}, and \eqref{obs4}, we obtain the following asymptotics.

\begin{corollary}\label{newsec4lem3}
 Let $\{a_d(n)\}_{n\ge 0}$ be as in \eqref{genexample}. We have, 
	\[
	\Delta^{[r]}_h(a_d)(n)\sim c_1(d)\left(h \l^*  A_{\l^*}(d)\right)^r \frac{\exp\left(\sum_{\l\in \mathcal{S}}  A_{\l}(d)n^{\l}\right) }{n^{c_2(d)+\left(1-\l^*\right)r}}\ \ \text{ as}\ \ n\to \infty.
	\]
\end{corollary}

\begin{proof}[Proof of \Cref{mshiftthm}]
By Wright \cite{Wright}, we have 
\begin{equation*}
\textup{pp}(n)\sim C_1\frac{e^{A n^{\frac 23}}}{n^{\frac{25}{36}}}\ \ \text{as}\ \ n\to \infty,
\end{equation*}
where $C_1$ and $A$ are as in \Cref{mshiftthm}. Applying \Cref{newsec4lem3}, the claim follows.\qedhere 
\end{proof}

 Using \Cref{newsec4lem3}, \eqref{pnasymp}, and \cite{HR}, we obtain the following asymptotics.

\begin{corollary}\label{shiftediffthm}
	We have, as $n\to \infty$, 
\begin{align*}
	\Delta^{[r]}_{h}(p)(n)&\sim
\frac{\pi^r h^r}{2^{\frac r2+2}3^{\frac{r+1}2}}\frac{ e^{\pi\sqrt{\frac{2}{3}n}}}{n^{\frac r2+1}},\qquad 
\Delta^{[r]}_{h}(\overline{p})(n)\sim \frac{\pi^rh^r}{2^{r+3}} \frac{e^{\pi\sqrt{n}}}{n^{\frac r2+1}}.
\end{align*}
\end{corollary}

\begin{proof}[Proof of \Cref{Jensencor1}]
 Using \eqref{deltadef}, \Cref{shiftediffthm}, and \Cref{sec:Jensenlem1}, we conclude the proof.
\end{proof}

\begin{proof}[Proof of \Cref{Lagthm1}]
Using \Cref{newcor1} and \Cref{shiftediffthm}, we conclude the proof.
\end{proof}

\section{Proofs of Theorems \ref{InflogconcaveThm}\textendash\ref{logdiffconj1}}\label{newsec2}

\begin{proof}[Proof of \Cref{InflogconcaveThm}]
	Using \eqref{logdef} and \eqref{Lagineq}, we have  
	\begin{equation}\label{Inflogconcaveeqn1}
	\frac{\mathcal{L}^{[1]}(a)(n+1)}{a^2(n+1)}
	=\frac{L_1(a)(n)}{a^2(n+1)}.
	\end{equation}
	Since $\{a(n)\}_{n\ge 0}$ satisfies 
	\ref{GenGORZThm1assump1} with $|j|\le 2$, \ref{GenGORZThm1assump2}, and \ref{genthmassump2}, we have, by  \Cref{newlem1}, 
	\[
	\frac{L_1(a)(n)}{a^2(n+1)}\sim 
	2\d^2(n)\ \ (\text{as}\ n\to \infty).
	\]
	Plugging this into \eqref{Inflogconcaveeqn1} and using \ref{InflogconcaveThmassump1} with $j=\pm 1$, \ref{InflogconcaveThmassump2}, and \ref{InflogconcaveThmassump3}, we obtain
	\begin{equation}\label{Inflogeqn5}
	\mathcal{L}^{[1]}(a)(n)\sim 2\d^2(n)a^2(n)\ \ (\text{as}\ n\to \infty).
	\end{equation}
	
	We now prove the theorem by induction on $r\in \mathbb{N}$. Using \eqref{Inflogeqn5}, it holds for $r=1$. Assume next that the theorem holds for $r$. Define $a_r(n):=\mathcal{L}^{[r]}(a)(n)$. By induction hypothesis,
	\begin{equation}\label{Indhyp}
	a_r(n)\sim  2^{2^{r+1}-r-2}\d(n)^{2^{r+1}-2} a^{2^r}\!(n). 
	\end{equation}
 Therefore, using \eqref{Indhyp}, \Cref{Fact1}, \ref{GenGORZThm1assump1}, and \ref{InflogconcaveThmassump1}, 
	$\{a_r(n)\}_{n\ge 0}$ satisfies \ref{GenGORZThm1assump1} with $j=\pm 1$, \ref{GenGORZThm1assump2}, and \ref{genthmassump2} with $A_r(n):=(2^{r+1}-2)B(n)+2^rA(n)$ and $\d_r(n):=\sqrt{2^r\d^2(n)-\left(2^{r+1}-2\right)\b^2(n)}$. Note that $\d_r(n)\sim 2^{\frac r2}\d(n)$ as $n\to \infty$. Moreover, employing \Cref{Fact1}, $\{\d_r(n)\}_{n\ge 0}$ satisfies \ref{InflogconcaveThmassump1} with $j=\pm 1$, \ref{InflogconcaveThmassump2}, and \ref{InflogconcaveThmassump3}. Using \eqref{Inflogeqn5} and \eqref{Indhyp}, the theorem holds for $r+1$.\qedhere
\end{proof}

Using \Cref{InflogconcaveThm}, we obtain the asymptotic main term of $\mathcal{L}^{[r]}(a_d)(n)$.

\begin{lemma}\label{Inftlogconcave:lem}
Let $\{a_d(n)\}_{n\ge 0}$ be as in \eqref{genexample}. For $r\in \mathbb{N}$, we have,  as $n\to \infty$,
	\[
	\mathcal{L}^{[r]}(a_d)(n)\sim 2^{2^r-r-1}c_1^{2^r}(d)\left(\l^*\left(1-\l^*\right)A_{\l^*}(d)\right)^{2^r-1}\frac{\exp\left(2^r\sum_{\l\in \mathcal{S}}A_{\l}(d)n^{\l}\right)}{n^{2^rc_2(d)+(2^r-1)(2-\l^*)}}.
	\]
\end{lemma}

\begin{proof}
Using \Cref{fact2} \ref{fact2part1}, $\{a_d(n)\}_{n\ge 0}$ satisfies \ref{GenGORZThm1assump1} with $|j|\le 2$, \ref{GenGORZThm1assump2}, and \ref{genthmassump2}. Moreover, by \Cref{fact2} \ref{fact2part2}, $\{\d(n)\}_{n\ge 0}$ satisfies \ref{InflogconcaveThmassump1} with $j=\pm 1$, \ref{InflogconcaveThmassump2}, and \ref{InflogconcaveThmassump3}. Using \Cref{InflogconcaveThm}, \eqref{impfact}, and \eqref{genexample}, we conclude the proof.\qedhere
 \end{proof}

\begin{proof}[Proof of \Cref{inflogthm1}]
 Applying \Cref{Inftlogconcave:lem} to \eqref{pnasymp}, the claim follows.\qedhere 
\end{proof}

 Next, we study the asymptotics of $\operatorname{T}^{[r]}_a(n)$. First, by the Leibniz formula, we have 
\begin{equation}\label{det}
\operatorname{T}^{[r]}_a(n)
=\sum_{\sigma\in S_r}\text{sgn}(\s)\prod_{m=1}^{r}a\left(n-m+\s(m)\right),
\end{equation} 
where $S_r$ is the symmetric group of $r$ letters and $\text{sgn}(\s)$ denotes the sign of the permutation $\s$. Next, we require an elementary lemma.

\begin{proposition}\label{propdet}
	For $r\in \mathbb{N}_{\ge 2}$, we have 
	\[
	\prod_{m=1}^{r}\left(1-e^{-m x}\right)\sim r! x^r\ \text{as}\ x\to 0.
	\]	
\end{proposition}

\begin{proof}[Proof of \Cref{detthm}]
	By \eqref{det} and \ref{GenGORZThm1assump1} with $|j|\le r-1$, we have, as $n\to \infty$,
	\begin{align*}
	&\frac{\operatorname{T}^{[r]}_a(n)}{a^{r}(n)}\sim \det\left(\left(e^{-\left(k-\ell\right)^2\d^2(n)}\right)_{1\le k,\ell\le r}\right).
	\end{align*}
	Thus, to prove the theorem, it suffices to show that, as $n\to \infty$, 
	\begin{equation}\label{detthm3}
	\mathcal{T}_r(n):=\det\left(\left(e^{-\left(k-\ell\right)^2\d^2(n)}\right)_{1\le k,\ell\le r}\right)\sim \frac{(r-1)!^r2^{\frac{r(r-1)}{2}}}{\prod_{j=2}^{r-1}j^{j}}\d^{r(r-1)}(n).
	\end{equation}
	To prove this, we use induction on $r\in \mathbb{N}_{\ge 2}$. For $r=2$, \eqref{detthm3} holds by \ref{GenGORZThm1assump2}. Assume next that \eqref{detthm3} holds for some $r\ge 2$. From \cite{DP}, we have that
	\begin{align*}
	\mathcal{T}_{r+1}(n)
	&
	=\mathcal{T}_r(n)\prod_{m=1}^{r}\left(1-e^{-2m\d^2(n)}\right).
	\end{align*}
 Using \Cref{propdet} and the induction hypothesis then gives that \eqref{detthm3} holds for $r+1$.\qedhere 
\end{proof}

 Using \Cref{fact2} \ref{fact2part1}, \Cref{detthm}, \eqref{impfact}, and \eqref{genexample}, we obtain the following.

\begin{corollary}\label{detmain1}
	Let $\{a_d(n)\}$ be as in \eqref{genexample}. As $n\to \infty$, we have
	\[
	\operatorname{T}^{[r]}_{a_d}(n)\sim \frac{(r-1)!^r\left(\l^*(1-\l^*)A_{\l^*}(d)\right)^{\frac{r(r-1)}{2}}c^r_1(d)}{\prod_{j=2}^{r-1}j^{j}}\frac{\exp\left(r\sum_{\l\in \mathcal{S}}A_{\l}(d)n^{\l}\right)}{n^{\left(1-\frac{\l^*}{2}\right)r(r-1)+c_2(d)r}}. 
	\]
\end{corollary}

\begin{proof}[Proof of \Cref{detpn}]
 Applying \Cref{detmain1} with \eqref{pnasymp}, we conclude the proof.\qedhere 
\end{proof}

\begin{proof}[Proof of \Cref{detopn}]
 Applying \Cref{detmain1} with the asymptotics of $\overline{p}(n)$, we conclude the proof. \qedhere 
\end{proof}

\begin{proof}[Proof of \Cref{detbpn}]
Applying \Cref{detmain1} with \Cref{brokenlem1}, we obtain the first claim. For the remaining claims, we apply \Cref{detmain1} and \Cref{shiftediffthm}.\qedhere 
\end{proof}

\begin{proof}[Proof of \Cref{itlagthm1}]
We use induction on $k\in \mathbb{N}_0$.  The base case $k=0$ follows from \Cref{newlem1}. Assume next that the statement holds for $k\in \mathbb{N}_0$. Define $a_{k}(n):=L^{[k]}_1\left(L_r(a)\right)(n)$. Using \eqref{itlagdef}, we write $L^{[k+1]}_1(L_r(a))(n)=L^{[1]}_1(a_k(n))$.
By the induction hypothesis, 
\begin{align}\label{itlageqn2}
a_k(n)&\sim 
(2r-1)!!^{2^k}2^{ 2^{k+1}\left(r+1\right)-k-3}\d(n)^{2^{k+1}\left(r+1\right)-2} a^{2^{k+1}}(n)\ \ (\text{as}\ n\to \infty).
\end{align}
 Therefore for $j=\pm 1$, by \Cref{Fact1}, \ref{InflogconcaveThmassump1}, and \ref{GenGORZThm1assump1}, we have as $n\to \infty$,
\begin{align*}
\log\left(\frac{a_k(n\pm 1)}{a(n)}\right)&=\pm\left(\left(2^{k+1}\left(r+1\right)-2\right)B(n)+2^{k+1}A(n)\right)\\
&\hspace{4 cm}-\left(2^{k+1}\d^2(n)-\left(2^{k+1}\left(r+1\right)-2\right)\b^2(n)\right)+o(1).
\end{align*}
Thus  $\{a_k(n)\}_{n\ge 0}$ satisfies \ref{GenGORZThm1assump1} with $j=\pm 1$,  $A_k(n):=(2^{k+1}\left(r+1\right)-2)B(n)+ 2^{k+1} A(n)$, and $\d_k(n):=\sqrt{2^{k+1}\d^2(n)-(2^{k+1}(r+1)-2)\b^2(n)}$. Using \ref{InflogconcaveThmassump3} and \ref{GenGORZThm1assump2}, \ref{GenGORZThm1assump2} holds for $\{\d_k(n)\}_{n\ge 0}$. By \ref{genthmassump2} and \ref{InflogconcaveThmassump2}, $\{a_k(n)\}_{n\ge 0}$ satisfies \ref{genthmassump2}. Using \Cref{newlem1} and the induction hypothesis then gives that \eqref{itlageqn2} holds for $k+1$.\qedhere 
\end{proof}

Using \Cref{itlagthm1}, we obtain the asymptotic main term of $L^{[k]}_1(L_r(a_d))(n)$.

\begin{corollary}\label{itlagthm2}
Let $\{a_d(n)\}_{n\ge 0}$ be as in \eqref{genexample}. For $k\in \mathbb{N}_0$ and $r\in \mathbb{N}$, we have, as $n\to \infty$,
\begin{align*}
L^{[k]}_1\left(L_r(a_d)\right)(n)&\sim (2r-1)!!^{2^k} 2^{2^k(r+1)-k-2}\left(\l^*(1-\l^*)A_{\l^*}(d)\right)^{2^k(r+1)-1} c^{2^{k+1}}_1(d)\\
&\hspace{6.5 cm}\times \frac{\exp\left(2^{k+1}\sum_{\l\in \mathcal{S}}A_{\l}(d)n^{\l}\right)}{n^{\left(2-\l^*\right)\left(2^k(r+1)-1\right)+2^{k+1}c_2(d)}}.
\end{align*}
\end{corollary}
\begin{proof}
 By \Cref{fact2} \ref{fact2part1}, $\{a_d(n)\}_{n\ge 0}$ satisfies {\rm{\ref{GenGORZThm1assump1}}} with $|j|\le 2r$, {\rm{\ref{GenGORZThm1assump2}}}, and {\rm{\ref{genthmassump2}}}. By \Cref{fact2} \ref{fact2part2}, \ref{InflogconcaveThmassump1} with $|j|\le 1$, \ref{InflogconcaveThmassump2}, and \ref{InflogconcaveThmassump3} hold for $\{\d(n)\}_{n\ge 0}$. Thus, applying \Cref{itlagthm1}, \eqref{impfact}, and \eqref{genexample}, we conclude the proof.\qedhere 
\end{proof}

\begin{proof}[Proof of \Cref{itlagthm3}]
 Applying \Cref{itlagthm2} with \eqref{pnasymp}, the asymptotics of $\overline{p}$, and the asymptotics of $\textup{spt}(n)$ (see \cite{B}), we conclude the proof.\qedhere 
\end{proof}

Next, we show the asymptotics of $\Delta^{[r]}(g_d)(n)$. 

\begin{lemma}\label{logdifflem1}
Let $\{g_d(n)\}_{n\ge 0}$ be as in \eqref{loggenexample}. For $r\in \mathbb{N}$ and $0<h_2(d)<1$, we have,
	\[
		(-1)^{r+1}\Delta^{[r]}(g_d)(n)\sim \frac{h_1(d)h_2(d)\prod_{m=1}^{r-1}\left(r-m-h_2(d)\right)}{n^{r-h_2(d)}}\ \ \text{as}\ n\to \infty.
	\]
\end{lemma}

\begin{proof}
We prove the lemma by induction on $r\in \mathbb{N}$. Using \eqref{loggenexample} and \Cref{Fact1}, we see that $\{g_d(n)\}_{n\ge 0}$ satisfies \ref{GenGORZThm1assump1} with $|j|\le 1$, $A(n):=\frac{h_2(d)}{n}$, and  $\d(n):=\sqrt{\frac{h_2(d)}{2}}\frac 1n$, \ref{GenGORZThm1assump2}, and \ref{genthmassump2}. Consequently, using \Cref{newsec4lem1} (with $h=r=1$) and \eqref{loggenexample}, we have, as $n\to \infty$,
	\begin{equation*}
	\Delta^{[1]}\left(g_d\right)(n)
	\sim \frac{h_1(d)h_2(d)}{n^{1-h_2(d)}}.
	\end{equation*}
 Thus, for $r=1$, the lemma holds. Next, assume by induction, that the lemma holds for $r\in \mathbb{N}$. Define  $g^{[r]}_d(n):=(-1)^{r+1}\Delta^{[r]}\left(g_d\right)(n)$. 
	Consequently, by \Cref{Fact1} and the definition of $\Delta^{[r]}(a)(n)$, the lemma holds for $r+1$.\qedhere
\end{proof}

The following lemma gives the asymptotics of $\Delta^{[r]}(\log(a_d))(n)$.

\begin{lemma}\label{logdifflem2}
	Let $\{a_d(n)\}_{n\ge 0}$ be as in \eqref{genexample}. Then we have, as $n\to \infty$,
	\[
	(-1)^{r+1}\Delta^{[r]}\left(\log(a_d)\right)(n)\sim \frac{\l^*A_{\l^*}(d)\prod_{m=1}^{r-1}\left(r-m-\l^*\right)}{n^{r-\l^*}}.
	\]
\end{lemma}
\begin{proof}
Note that if $f(n)\sim g(n)$, then $\log (f(n))= \log(g(n))(1+o(1))$ as $n\to \infty$ unless $g(n)\to 0$. Thus, using \eqref{genexample}, \ref{genexampleassump2}, and \ref{genexampleassump3}, we have, as $n \to \infty$, 
	\begin{align*}\nonumber 
	&\log(a_d(n))
	\sim A_{\l^*}(d)n^{\l^*}.
	\end{align*}
 Finally, using \ref{genexampleassump2}, \ref{genexampleassump3}, and  \Cref{logdifflem1}, we conclude the proof.\qedhere 
\end{proof}

\begin{proof}[Proof of \Cref{logdiffconj1}]
 Using \Cref{brokenlem1} and \Cref{logdifflem2}, we conclude the proof.\qedhere
\end{proof}

\section{Question for future research}\label{concluding}
We end with a few questions and remarks on potential future work.  
\begin{remarks}
\leavevmode\par
\begin{enumerate}[leftmargin=*, labelsep=0.5em]
 \item \noindent  In \Cref{genthm}, we assume {\rm{\ref{GenGORZThm1assump1}}}, {\rm{\ref{GenGORZThm1assump2}}}, and {\rm{\ref{genthmassump2}}}. Can one relax (or modify) {\rm{\ref{genthmassump2}}} to obtain more general results?
\item With additional assumptions, can one obtain secondary order asymptotics of $I^{[\bm{M}]}_{\bm{\rho},\bm{s}}(a)(n)$ as studied in Theorem~\ref{genthm}?
\item In this paper, we restrict to sequences $\{a(n)\}_{n \ge 0}$ with asymptotics of the shape \eqref{genexample}. Can one extend our results to combinatorial sequences whose asymptotics do not have this shape?
\item While condition {\rm{\ref{GenGORZThm1assump1}}} in Theorem~\ref{GenGORZcor1} underlies the ``generic'' situation of Hermite polynomial limits for (renormalized) Jensen polynomials, there are many other interesting phenomena for other combinatorial sequences. 
For instance, in \cite{GriffinSouth}, Griffin and South obtained another distinguished class of polynomials that encode the asymptotics of Jensen polynomials for Taylor coefficients of entire functions. Can a suitable framework be developed for ``non-Hermite'' phenomena like that uncovered by Griffin and South? 
\item 
Many of the results in this paper can be shown to follow for large classes of sequences including Fourier coefficients of weakly holomorphic modular forms. For instance, we show that \eqref{Briggsineq} eventually  holds 
for a sequence  of the shape  \eqref{genexample}. Weakly holomorphic modular forms automatically satisfy such conditions. It would be interesting to pin down which of our results hold for these, and for other natural classes of automorphic objects. 
\end{enumerate}
\end{remarks}

\end{document}